\documentclass{amsart}
\usepackage{xcolor}
\newtheorem{theorem}{Theorem}[section]
\newtheorem{lemma}[theorem]{Lemma}
\newtheorem{proposition}[theorem]{Proposition}

\theoremstyle{definition}
\newtheorem{definition}[theorem]{Definition}

\theoremstyle{remark}

\numberwithin{equation}{section}

\def\Om{\Omega}
\def\C{\mathbb C^n}
\def\al{\alpha}

\def\Om{\Omega}
\def\la{\lambda}
\def\aij{a^{i\overline{j}}}

\newcommand{\rr}{\mathbb{R}}
\newcommand{\cc}{\mathbb{C}}

\newcommand{\ii}{\mathfrak{i}}

\newcommand{\de}{\partial}

\newcommand{\xn}{||(z',y_n)||^2}

\def\la{\lambda}
\def\Om{\Omega}
\title[Krylov approach for complex PDEs]{A priori estimates for the complex Monge-Amp\`ere equation after Krylov}
\begin{document}

\newcounter{remark}
\newcounter{theor}
\setcounter{remark}{0}
\setcounter{theor}{1}
\newtheorem{claim}{Claim}
\newtheorem{corollary}[theorem]{Corollary}
\newtheorem{question}{Question}[section]
\numberwithin{equation}{section}

	\author{S\l awomir Dinew}
	\address{Faculty of Mathematics and Computer Science,  Jagiellonian University 30-348 Krakow, Lojasiewicza 6, Poland}
	\email{slawomir.dinew@im.uj.edu.pl}
	\author{Marcin Sroka}
	\address{Faculty of Mathematics and Computer Science,  Jagiellonian University 30-348 Krakow, Lojasiewicza 6, Poland}
	\email{marcin.sroka@im.uj.edu.pl}

	\subjclass[2020]{}	
	\date{}
	
	\begin{abstract}We establish an analytic proof for the Krylov $C^{1,1}$ estimates for solutions of degenerate complex Monge-Amp\`ere equation. We also provide an analytic proof of the Bedford-Taylor interior  $C^{1,1}$ estimate.
	\end{abstract}
	
	\maketitle
	
	\section{Introduction}
	
	In \cite{CNS85} Caffarelli, Nirenberg and Spruck established the classical solvability of the Dirichlet problem for a general {\it real} nonlinear elliptic equation of Hessian type under the natural geometric assumptions {on} the boundary of the domain considered.
	\begin{theorem}[\cite{CNS85}] \label{cns}
	Let $U$ be a smoothly bounded domain in $\mathbb R^n$ and 
	\[ F(D^2u)=f(\lambda(D^2(u)))\] 
	be an elliptic operator of Hessian type, with ellipticity with respect to the cone $\Gamma$. Assume the conditions (\ref{below}) below {hold}.
	Let $g$ be smooth, up to the boundary of $U$, and \[ g\geq c_0>0.\] Assume further the following convexity property of $\partial U$:
there is $R>0$ such that for any $x_0\in\partial U$ if $\kappa_j,\ j=1,\cdots,n-1$ denote the principal curvatures in $x_0$ with respect to the inner normal, then the vector$(\kappa_1,\cdots,\kappa_{n-1},R)$ belongs to $\Gamma$. Then, the {Dirichlet} problem
	\begin{equation} \label{cns}
	\begin{cases} u\in C^{\infty}(\overline{U})\ {\rm and}\ \forall_{\substack{x \in \overline{U}}} \lambda(D^2u)(x)\in\Gamma,\\
	 F(D^2u(x))=g(x)\ {\rm in}\ U,\\
	 u=\varphi\ {\rm on}\ \partial U
	\end{cases}
	\end{equation}
	admits unique solution for any boundary data $\varphi\in C^{\infty}(\partial U)$.
	\end{theorem}

	We wish to remark that similar results can also be extracted from the more general theory developed by Krylov - see \cite{Kr83, Kr84a, Kr84b, Kr90}.

	Typical examples of operators $F$ for which all required conditions are met are given by \[ f(\lambda)=\sigma_k(\lambda)^{1/k},\] 
	where 
	\[ \sigma_k(\lambda):=\sum_{1\leq i_1<\cdots<i_k\leq n}\lambda_{i_1}\cdots\lambda_{i_k}\] is the $k$-th elementary symmetric polynomial and the corresponding cone $\Gamma_k$ is given by
	\[ \Gamma_k:=\lbrace\lambda\in\mathbb R^n|\ \sigma_j(\lambda)>0,\ j=1,\cdots,k\rbrace.\]

	The analogue {of} this result for the complex Monge-Amp\`ere equation was established in \cite{CKNS86}.
	\begin{theorem}[\cite{CKNS86}]
	 Let $\Omega\subset\mathbb C^n$ be a $C^4$ regular bounded strictly pseudoconvex domain. Then the Dirichlet problem
	 \begin{equation}
	  \label{ckns} \begin{cases}
	 u\in PSH(\Omega)\cap C(\overline\Omega),\\
	    det(\frac{\partial^2 u}{\partial z_j\partial\bar{z}_k})=g(z),\\
	    u=\varphi\ {\rm on}\ \partial\Omega
	   \end{cases}\end{equation}
	   admits a unique solution $u\in C^2(\overline\Omega)$ provided $\varphi\in C^4(\partial\Omega)$ and $g\geq c_0>0$ is $C^2$ smooth on $\overline\Omega$.
	\end{theorem}
As both theorems play a fundamental role in geometric analysis, they were subject to generalizations in terms of the {\it optimality} of the assumptions. Of particular interest is the constant $c_0$, which, roughly speaking, controls (in a very weak way) the strict ellipticity of the equation. It is in particular well known (cf. \cite{CNS86} and references therein), that once the non-degeneracy assumption $g \geq c_0>0$ is dropped the {\it optimal} regularity one can hope for is $u \in C^{1,1}(\overline{U})$.

In \cite{Kr90} N. V. Krylov has shown this optimal $C^{1,1}$ regularity using probabilistic methods. Although his results hold for much more general equations, we state it only for the complex Monge-Amp\`ere operator as this is the subject of present work.
\begin{theorem}[Krylov]\label{krylovest}
Assume that $\Omega\subset\mathbb C^n$ is a $C^{3,1}$ regular bounded strictly pseudoconvex domain. Assume that $\varphi\in C^{3,1}(\partial\Omega)$ and that $g^{1/n}\in C^{1,1}(\overline{\Omega})$ is non-negative. Then the Dirichlet problem (\ref{ckns}) admits a unique solution that is $C^{1,1}$ smooth up to the boundary.
\end{theorem}

In particular it follows from Krylov's theorem that the solution $u$ is a {\it strong} one: off a set of measure zero it admits second-order Taylor expansion, and plugging the corresponding second order partial derivatives into the equation one ends up with the prescribed right hand side. For a thorough explanation of Krylov's probabilistic approach, we refer to \cite{Del}.

Since then, it has been an open question how to prove this regularity with analytic tools for complex equations, cf. \cite{DGZ, CMA}. In the real setting, this was achieved by Krylov himself in a series of papers \cite{Kr94a, Kr94b, Kr95a, Kr95b}. We also refer to the book \cite{Kr87} for many geometrical and probabilistic motivations behind this analytical approach.

The Krylov approach in the real case has been somewhat simplified in \cite{ITW} where the differentiation technique with respect to vector fields with skew-symmetric Jacobians was introduced.

Our project was initiated with the goal of extracting these real arguments from \cite{ITW} that do not have a complex analogue and to understand in detail the difficulties that arise in the analytic approach to the complex Krylov theorem. To our surprise, we found that modulo significant technical difficulties all the reasoning in \cite{ITW} can be carried for the Monge-Amp\`ere operator in the complex setting. Hence, we obtained the following theorem that resolves the aforementioned open problem.

\begin{theorem}\label{complexKrylov}
 There is an entirely analytic proof of Krylov's $C^{1,1}$ estimate for the complex Monge-Amp\`ere equation, i.e. of Theorem \ref{krylovest}.
\end{theorem}

Exactly as in \cite{ITW}, the proof consists of three parts.

The first is to obtain a gradient bound and reduce the Hessian bound to the boundary normal-normal derivative estimate. All this is already well-known to experts in the field.

The second one is a {\it weakly interior second order estimate} stating that if we denote 
\[ \Omega_{\delta}:=\lbrace z\in \Omega\ |\ dist(z,\Omega)>\delta\rbrace \] 
then, for $\delta>0$ sufficiently small, $\Omega_{\delta}$ is also $C^{3,1}$ smooth, and one has the bound:
\begin{equation}\label{weaklyint}
 \forall_{\varepsilon\in (0,1)} \exists_{C_{\varepsilon,\delta}>0} \forall_{z\in\Omega_{\delta}}\ |D^2u(z)|\leq\varepsilon sup_{w\in\partial\Omega}|D^2u(w)|+C_{\varepsilon,\delta}
\end{equation}
for a constant $C_{\varepsilon,\delta}$ depending on many additional quantities but not on $||D^2u||$. The proof of this fact, modeled on the argument of Ivochkina \cite{Iv}, is presented in Section \ref{wint}. We recall that despite its complexity, this part of the argument does not require substantially new ideas in the complex setting. The only interesting technicality is that, compared to \cite{Iv}, there is a novel term to handle\footnote{It is the third part of the $A_1$ term in Step 1 of the computation.}, but as it turns out it is harmless. In the special case of the complex Monge-Amp\`ere operator such weak interior bounds were established in \cite{Zh} using probabilistic methods (the gradient bound was established earlier in \cite{Kr93}). The analytic proof is nevertheless of its own interest as it in particular generalizes the Bedford - Taylor bound from \cite{BT76} and, in particular, answers Question 25 from \cite{DGZ}.

The third part is to exploit the weakly interior estimate and other a priori bounds to establish control {on} the boundary normal-normal derivatives in terms of already contained quantities. To this end Ivochkina, Trudinger and Wang exploit a specially constructed almost tangential vector fields with many useful properties for further calculations, cf. Section 2 in \cite{ITW}. The reason is as follows: whenever the equation is differentiated twice with respect to the vector field $\tau$, numerous uncontrollable error terms occur unless $\tau$ has a special structure, as Lemma 2.1 in \cite{ITW} shows. The crucial part of the argument is that such almost tangential skew-symmetric vector fields can always be found, at least in a neighborhood of any boundary point, cf. equation $(2.9)$ in \cite{ITW}.

Although it is rather obvious what the complex analogue of such $\tau$'s should be, and also their good properties when it comes to differentiating complex Hessian operators are also not unexpected - see Proposition \ref{fquadratic}, their very existence near a boundary point is in general {\bf obstructed} for complex admissible domains. However, the crucial observation is that in special {\it adapted coordinates} the existence of the aforementioned vector fields is {\bf unobstructed} and the argument from \cite{ITW} can be completed. The existence of such coordinates, in turn, given by {Proposition \ref{bedtay}}, requires a non-linear holomorphic change of variables, which makes it applicable only for the complex Monge-Amp\`ere operator.

In relation to the last remark, we wish to point out that a vast majority of the arguments below, with the notable exception of the mentioned existence of adapted coordinates, apply verbatim to more general complex Hessian equations. For this reason, we formulate the majority of our technical results in the more general language of complex Hessian equations. Whether the full analytic proof of Krylov's theorem can be established for complex Hessian equations (under the natural assumptions) remains an open question. The above-described obstruction shows that the present argument has to be modified, at best. We hope to address this problem in future work.

The note is organized as follows: in Section 2 we gather the basic notation, definitions and classical results that will be needed later on. Section 3 contains some standard reductions and PDE facts that we shall implicitly use later on. {Section \ref{gradint} is devoted to the proof of the weakly interior gradient estimate - see Theorem \ref{1Zhou}. The mentioned weakly interior second order bound Theorem \ref{weakinteriorbound} is proved in Section \ref{hessint} utilizing in particular many computations from Sections 3 and \ref{gradint}. The main argument is contained in Section 5 where we show the main result Theorem \ref{complexKrylov}.} In the last section we shall discuss to what extent the argument can be applied for a general elliptic complex Hessian equation.

{\bf Acknowledgments}: The first named author was partially supported by Sheng grant no. 2023/48/Q/ST1/00048 of the National Science Center, Poland. The research of the second author was supported in part by National Science Center of Poland grant no. 2021/41/B/ST1/01632.

\section{Preliminaries}

\subsection{Notation} As is customary, we shall denote by $C$ various constants that do not depend on the pertinent quantities in a given estimation. In particular, these $C$'s may vary from line to line. In the main argument, we furnish these constants with additional labels so that the dependencies may be more easily tracked.

We shall freely use Einstein's notation of summation over repeated indices.

For any complex vector fields $X$, $Y$, $(1,0)$ vector fields $\xi$, $\eta$ and smooth function $v$ in $\cc^n$ we use the notation
\begin{equation} \begin{gathered}
v_X=dv(X), \\ v_{(\xi)}=dv\left( \xi_l\frac{\de}{\de z_l} +\overline{\xi}_l\frac{\de}{\de \overline{z}_l} \right)=\xi_lv_{z_l} +\overline{\xi}_lv_{\overline{z}_l}.
\end{gathered} \end{equation}
We remark that for the real vector field $\Re(\xi):=\frac12(\xi_l\frac{\de}{\de z_l} +\overline{\xi}_l\frac{\de}{\de \overline{z}_l})$ one has
\begin{equation}\label{realtocomplex}
 2v_{\Re(\xi)}=v_{(\xi)}.
\end{equation}

For the second order derivatives we use:
\begin{equation}\begin{gathered}
v_{XY}=X^iY^ju_{ij}+X^iY^{\bar{j}}u_{i\bar{j}}+X^{\bar{i}}Y^ju_{\bar{i}j}+X^{\bar{i}}Y^{\bar{j}}u_{\bar{i}\bar{j}},\\
v_{[\xi][\eta]}=v_{\xi \eta}+v_{\xi \overline{\eta}}+v_{ \eta \overline{\xi}}+v_{\overline{\xi} \overline{\eta}},\\
v_{(\xi) (\eta)}=\left(v_{(\xi)}\right)_{(\eta)}.
\end{gathered}\end{equation}

Again, we note that
\[v_{(\xi) (\eta)}=4\big(u_{\Re (\xi)}\big)_{\Re (\eta)}.\]

We remark that our bracket notation follows \cite{ITW} and is very different from Krylov's or the one from \cite{Zh}. We also note that we follow the standard notation for tensor components, with respect to the basis $\partial_{z_i}$, $\partial_{\overline{z_i}}$ in the complex case and $\partial_{x_i}$, $\partial_{y_i}$ in the real one. The exception is that for the $(1,0)$ vector field we use lower indices.

By $\Omega$ we shall denote a bounded domain in $\mathbb C^n$. Unless otherwise stated we shall assume that $\partial\Omega$ is $C^{3,1}$ regular\footnote{Starting from Subsection 3.2 $C^4$ regularity will be the standing assumption.} and strictly pseudoconvex. We will crucially exploit the notion of a defining function for $\Omega$ in the following sense.
\begin{definition} \label{pseudoconv} The function $\psi$ is defining function off class $C^k$ for $\Omega$ provided it is defined in a neighborhood of $\Omega$, $\psi\in C^{k}$, $\Omega=\lbrace \psi>0\rbrace$, $\partial\Omega=\lbrace \psi=0\rbrace$ and $||\nabla\psi_{|\partial\Omega}||\geq 1$.
\end{definition}
Modifying $\psi$, if necessary, we can and assume that $-\psi$ is strictly plurisubharmonic up to the boundary of $\Omega$, that is, the complex Hessian of $-\psi$ has its eigenvalues uniformly positive on $\overline{\Omega}$. For the more general situation of strictly $\Gamma$- admissible domains for pairs $(F,\Gamma)$ satisfying (\ref{below}) one can find $\psi$ such that $-\psi$ is strictly $\Gamma$ - admissible.

\subsection{Complex Hessian operators} 
We {collect} here the basics of the theory of Hessian operators as presented, for example, in \cite{CNS85}. Consider an open convex cone $\Gamma\subset\mathbb R^n$ with a vertex at the origin. We shall assume that
\begin{equation}\begin{gathered}\label{cone}
 \Gamma\ {\rm is\ symmetric\ w.r.t.\ permutation\ of\ coordinates},\\
 \Gamma_n\subset\Gamma\subset\Gamma_1,
\end{gathered}\end{equation}
where $\Gamma_j:=\lbrace\la\in\mathbb R^n\ |\ \sigma_k(\la)>0,\ k=1,\cdots,j \rbrace$. In particular $\Gamma_n$ is the positive orthant.

\begin{definition}\label{admis}
A $C^2$ function $u:\Omega\longmapsto\mathbb R$ is said to be $\Gamma$ admissible, for a fixed cone $\Gamma$ satisfying (\ref{cone}),
if at each $z\in\Omega$ the vector $\lambda(D^2_{\mathbb C}u)(z)$ of eigenvalues of complex Hessian
\[D^2_{\mathbb C}u=\left( \frac{\partial^2u}{\partial z_j\partial\bar{z}_k}(z)\right)\] of $u$ belongs to $\overline{\Gamma}$. Moreover, $u$ is said to be strictly admissible if for any $\varphi\in C^2_{0}(\Omega)$ there is an $\epsilon>0$ such that $u+\epsilon\varphi$ is $\Gamma$-admissible.
\end{definition}
We note that due to the symmetry of $\Gamma$, the order of the eigenvalues is irrelevant. Another basic observation is that, by (\ref{cone}) any $C^2$ $\Gamma$-admissible function is necessarily subharmonic.

\begin{definition}\label{tilgamma}
 Let $\Gamma$ be a cone as in (\ref{cone}). By $\tilde{\Gamma}$ we shall denote the convex cone of the Hermitian matrices $A$ such that the corresponding vector of eigenvalues $\lambda(A)$ belongs to $\Gamma$.
\end{definition}

We observe that a function $u$ is admissible if and only if $D^2_{\mathbb C}u$ belongs to the Euclidean closure of $\tilde{\Gamma}$.
To any such $\tilde{\Gamma}$ one can associate a class of operators which are called complex Hessian operators.
\begin{definition}\label{hes}
 A complex Hessian operator is a function \[ F:\tilde{\Gamma}\longmapsto\mathbb R,\] that depends only on the eigenvalues of the matrices from $\tilde{\Gamma}$ i.e. there is a function $f:\Gamma\longmapsto\mathbb R$, symmetric with respect to permutations of coordinates such that
 $$F(A)=f(\lambda(A)).$$
\end{definition}

The {\it ellipticity} of such an operator can be defined in a weak sens as follows:
\begin{definition}\label{ellip}
 A complex Hessian operator $F$ defined on a cone $\tilde{\Gamma}$ is elliptic if for any $A\in \Gamma$ and any $P\in\Gamma_n$ one has
 $$F(A+P)\geq F(A).$$
\end{definition}
Next, we impose natural conditions on $F$:
\begin{equation}\label{below}
\end{equation}
	\begin{enumerate}
	 \item {\bf ellipticity}:\ $\forall \lambda \in\Gamma\ \forall i\in\lbrace1,\cdots,n\rbrace\ \frac{\partial f}{\partial \lambda_i}>0$;
	 \item {\bf concavity}:\ $f$ is concave on $\Gamma$;
	 \item {\bf homogeneity}:\ $\forall t>0\ \forall\lambda\in\Gamma$ $f(t\lambda)=tf(\lambda)$;
	 \item {\bf improved monotonicity}:\ for any compact $K$ in $\Gamma$ and any $C>0$ there is $R_0=R_0(C,K)>0$ such that for any $R\geq R_0$
	 \[ f(\lambda_1,\cdots,\lambda_{n-1},\lambda_n+R)\geq C;\]
	\item{\bf asymptotics}: \ for any compact $K$ in $\Gamma$ and any $C>0$ there is $R_0=R_0(C,K)>0$ such that for any $R\geq R_0$ 
	\[ f(R\lambda)\geq C\] and \[\limsup\limits_{\lambda\longmapsto\lambda_0\in\partial \Gamma}f(\lambda)=0.\]
	\end{enumerate}

The basic examples of pairs $(f,{\Gamma})$ that meet all these requirements are given by $((\sigma_k)^{1/k},\Gamma_k)$ for $k=1,\cdots,n$. We refer to \cite{HL} for a much broader set of examples.

Note that \[ F(D^2_{\mathbb C}u(z))\geq 0\] for any $\Gamma$-admissible function $u$ and any $z\in\Omega$. The asymptotics condition coupled with the concavity implies that \[F(D^2_{\mathbb C}u(z))> 0\] for any strictly admissible function $u$.

As $\Gamma$-admissibility is invariant with respect to unitary change of coordinates, the following additional property of a general Hessian operator $F$ will be very useful:
\begin{equation}\label{unitary}
 \forall A\in\tilde{\Gamma}\ \forall B\in U(n)\ \ F(B^*AB)=F(A),
\end{equation}
where $B^*$ denotes the conjugate transposed matrix to $B$.

In this generality Hessian operators do not posses more symmetries, which is a key observation in the context of our main result as advertised in the Introduction. In the special case of the complex Monge-Amp\`ere operator we have the following additional symmetry following from the properties of the determinant operator:
\begin{proposition}\label{holo}
 Let $u$ be a $C^2$ smooth plurisubharmonic function. Let also $G: U\longmapsto \Omega$ be a holomorphic map from a domain $U \subset \mathbb C^n$ to $\Omega$. Then $u\circ F$ is plurisubharmonic on $U$ and 
 \[ \begin{gathered} \left[ \det \left(\frac{\partial^2 (u\circ G)}{\partial w_s\partial\bar{w}_t}\right)(w)\right]^{1/n}=\left[ \det \left(\frac{\partial^2 u}{\partial z_j\partial\bar{z}_k}\right)\left(G(w)\right)\right]^{1/n} \cdot |\det G'(w)|^{2/n},\end{gathered} \]
 where $G'$ denotes the complex Jacobian matrix of the mapping $G$.
\end{proposition}
We remark that if $G$ happens to be locally biholomorphic then $|\det G'|>0$ and hence the right hand side has the same regularity as 
\[ \left[ \det \left(\frac{\partial^2 u}{\partial z_j\partial\bar{z}_k}\right)(z)\right]^{1/n}.\] 
This additional symmetry plays a central role in the construction of the adapted coordinates of Section 5.

\subsection{Krylov maximum principle} 
In the proof of the weakly interior estimate we shall make use of the complex Krylov maximum principle. We state it in a way so that it becomes the complex analogue of Lemma 2.1 in \cite{Iv} or Lemma 3.1 in \cite{ITW}:
 \begin{theorem}\label{phomog}
 	Let $\Omega$ be a domain in $\cc$ and let $w,v: \Omega\times\mathbb C^{n+1}\ni (z,\zeta')\longmapsto \mathbb R$ be $C^2$ functions. Let $L$ be a degenerate elliptic linear operator 
 	$$L[w]=\sum_{i,j=1}^{2n+1} A^{i\bar{j}}w_{i\bar{j}}-\sum_{k=1}^{2n+1}\Re(b_kw_k)$$
 	 for some semi-positive Hermitian matrix valued function $A$ of dimension $(2n+1)\times(2n+1)$ and $\mathbb C^{n+1}$-vector valued function $b$. Then, if $L[v]<0$ and $v>0$ we have
 
\begin{equation} \begin{gathered} \label{krylovmax}	
		\forall (z,\zeta')\in\Omega^N\ \ \
		\frac{w}{v}(z,\zeta')\leq \max \Big\lbrace \sup\limits_{\Omega^N} \frac{L[w]}{L[v]}, \sup\limits_{\partial \Omega^N}\frac wv \Big\rbrace\\
		\leq 	\max \Big\lbrace \sup\limits_{\Omega^N} \frac{-L[w]}{-L[v]}, \sup\limits_{\partial \Omega^N}\frac wv \Big\rbrace,
\end{gathered}\end{equation}
with $\Omega^N:=\Omega\times\lbrace\zeta'\in\mathbb C^{n+1}|\ \frac12<||\zeta'||<2\rbrace$.
\end{theorem}
\begin{proof}
	We reproduce the argument in \cite{ITW} for the sake of completeness.
	
	If the maximum occurs on the boundary, we are through. Otherwise, at a maximum point we have
	$$0=\left(\frac{w}{v}\right)_i=\frac{w_iv-v_iw}{v^2}$$
	and
	$$0\geq L\left[\frac{w}{v}\right]=\frac{L[w]v+A^{i\bar{j}}w_iv_{\bar{j}}-A^{i\bar{j}}v_iw_{\bar{j}}-L[v]w}{v^2}=\frac{L[w]v-L[v]w}{v^2}.$$
Rearranging terms, we obtain the claimed result.
\end{proof}

\section{Initial adjustments}\label{initial}
In this section, we gather some standard PDE arguments related to the complex Monge-Amp\`ere and other complex Hessian equations.
\subsection{Basic reductions}
In the following, we initiate the analysis of the Dirichlet problem:
\begin{equation}\label{tobeginwith}
\begin{cases}F(D^2_{\mathbb C}u(z))=g(z),\\
u|_{\partial\Omega}=\varphi.
    \end{cases}
\end{equation}
For the well-posedness of (\ref{tobeginwith}) a necessary geometric assumption is the strict $\Gamma$ pseudoconvexity of $\Omega$, i.e. the assumption is that there exists a {negative} defining function $\rho$, {i.e. minus the one from Definition \ref{pseudoconv}}, of class $C^{3,1}$ in a neighborhood of {$\overline{\Omega}$} such that the vector of the eigenvalues of the complex Hessian of $\rho$ belongs to $\Gamma$ for any $z\in {\overline{\Omega}}$ - see \cite{Li} and \cite{CNS85} for the real analogue of this assumption. Note that for $\Gamma=\Gamma_n$ this notion is the same as strict pseudoconvexity.

First of all, we recall the standard fact that it suffices to work under the assumptions $\partial\Omega\in C^4$, $\varphi\in C^4(\partial\Omega)$ and $g\in C^2(\Omega)$. Indeed, approximating $(\Omega,\varphi,g)$ by $\Omega_m,\varphi_m,g_m$, where:
\begin{enumerate}
 \item $\Omega_m$ is an increasing sequence of $C^4$ smoothly bounded domains within $\Omega$ such that $\cup_{m=1}^{\infty}\Omega_m=\Omega$ (as $\Omega$ is assumed to have strictly $\Gamma$ admissible defining function the same applies to $\Omega_m$);
 \item $\varphi_m\in C^4(\partial\Omega_m)$ and $\varphi_m\longmapsto\varphi$ in $C^{3,1}$ norm\footnote{To be precise one compares $\varphi_m\circ T_m$ and $\varphi$ with $T_m$ being the diffeomorphism (existing for large $m$), which is the projection of $\partial\Omega_m$ onto $\partial\Omega$ using a retraction in a tubular neighborhood of $\partial\Omega$.};
 \item $g_m\in {C^2}(\overline{\Omega}_m)$ and $g_m\longmapsto g$ in ${C^{1,1}}$ norm;
\end{enumerate}
it suffices to bound the solution $u_m$ in $C^2$ norm independently of the modulus of continuity of fourth derivatives of $\varphi_m,\partial\Omega_m$ and second derivatives of $g_m$.

Another standard reduction is to replace $g$ by $g+\varepsilon$ for some small $\varepsilon>0$. Then \cite{CKNS86} (or \cite{Li} for a general Hessian equation) implies that the solution $u_{\varepsilon}$ exists and is of class $C^{2,\alpha}$ for some $\alpha$. Thus, the analysis is reduced to $\varepsilon$ independent (i.e. $\inf g$ independent) second order estimates for $u_{\varepsilon}$. Additionally, this approximation justifies sequential computations where otherwise division by zero could have appeared.

To wrap up, from now on we assume that:
\begin{equation}\label{Basicassumptions}
\end{equation}
\begin{enumerate}
 \item $\Omega$ {\rm is}\ $C^4$ {\rm smooth\ strictly} $\Gamma$-pseudoconvex;
 \item $\varphi\in C^4(\partial\Omega)$;
 \item $0<g\in {C^2}(\overline{\Omega})$.
\end{enumerate}

\subsection{Basic estimates}
In this subsection, we collect the standard estimates associated with complex Hessian equation satisfying the conditions (\ref{below}). These are all standard now, and we include them only for the sake of completeness.

\begin{proposition}\label{unif}
Let $\Omega$ be a $C^4$ smoothly bounded domain in $\mathbb C^n$. Let the Hessian operator $F$ and the associated cone $\Gamma$ be as in the preliminaries. Assume that for some $C^4$ smooth function $\varphi$ on $\partial\Omega$ and some $0\leq g\in C^2(\overline{\Omega})$ the $\Gamma$-admissible function $u$, continuous up to the boundary, solves
$$\begin{cases}F(D^2_{\mathbb C}u(z))=g(z),\\
u|_{\partial\Omega}=\varphi.
    \end{cases}$$
    Then there is a constant $C$ depending on $\varphi,g$ and $\Omega$ such that $$\max_{\overline \Omega}|u|\leq C.$$

\end{proposition}
\begin{proof}(Sketch) 
As $u$ is in particular subharmonic it is bounded from above by the harmonic extension $h_\varphi$ of the boundary data. For the lower bound, one applies the maximum principle with \[ v(z):=a(||z||^2-B),\] where $B$ is so large that $v<\varphi$ on $\partial\Omega$, whereas $a>1$ is taken so that $F(D^2_{\mathbb C}v(z))\geq g(z)$ on $\overline{\Omega}$.
\end{proof}
\begin{proposition}\label{c1}
With $\Omega, F, \Gamma, g, \varphi$ and $u$ as above, there is a constant $C$ such that
$$\max_{\overline{\Omega}}|Du|\leq C.$$
\end{proposition}
\begin{proof}(Sketch) 
Fix a constant real vector field $T$. The function 
\[ w(z):=u_{T}(z)+a||z||^2\] satisfies
 \[ \sum_{k,j=1}^n\frac{\partial F}{\partial u_{k\bar{j}}}\Big(D^2_{\mathbb C}u(z)\Big)\frac{\partial^2 w}{\partial z_k\partial\bar{z}_j}\geq 0\]
 for sufficiently large $a$ (depending on $C^1$ norm of $g$).

 As the linearized operator of $F$ is (degenerate) elliptic one has
 \[ \max_{\overline{\Omega}}|u_T|\leq C+\max_{\partial{\Omega}}|u_T|.\]

For the boundary estimate observe that $u$ is squeezed between the harmonic extension $h_\varphi$ of $\varphi$ and

\begin{equation} \label{syli}
	h_\varphi + a \rho
	\end{equation}
	where we recall that $\rho$ is strictly $\Gamma$-admissible {negative} defining function and $a$ is taken large enough. 
\end{proof}
\begin{proposition}\label{toboun}
 With the setting as above
 $$\max_{\overline{\Omega}}|u_{[T][T]}|\leq C+\max_{\partial\Omega}|u_{[T][T]}|.$$
\end{proposition}
\begin{proof}(Sketch) 
The subharmonicity of $u$ implies that it suffices to show the upper bound of $u_{[T][T]}$. This in turn follows from
\[ \sum_{k,j=1}^n\frac{\partial F}{\partial u_{k\bar{j}}}\Big(D^2_{\mathbb C}u(z)\Big)\frac{\partial^2 w}{\partial z_k\partial\bar{z}_j}\geq 0,\]
where this time
\[ w(z):=u_{[T][T]}(z)+a||z||^2\] for $a$ large enough. It is worth pointing out that in this proof the concavity of $F$ plays a role, as it allows one to control the third order terms in $u$ appearing after the second differentiation of the equation in direction $T$.
\end{proof}

\begin{proposition}\label{tantan}
With the setting as above and any local unit length real vector field $T$ tangential to the boundary, one has
\[ \max_{p\in\partial\Omega}|u_{[T][T]}(p)|\leq C.\]
\end{proposition}
\begin{proof}
This follows from the standard formula
\[ 0=T((u-\varphi)_{T})(p)=(u-\varphi)_{[T][T]}(p)+\kappa(p)(u-\varphi)_{\eta}.\]
 Here $\kappa(p,T)$ denotes the curvature in direction $T$.
\end{proof}
\begin{proposition}\label{tannor}
With the setting as above one has
\[ \max_{p\in\partial\Omega}|u_{[T][\eta]}(p)|\leq C.\]
\end{proposition}
\begin{proof}(Sketch)
 The argument is a modification of the one from \cite{Li}.
Fixing a point $p \in \partial\Omega$ and applying a (complex) affine change of coordinates, we can assume that $p$ is the coordinate origin and $\partial\Omega$ is locally given by the graph of $C^4$ function $\varphi(z',y_n)$ and $\Omega$ locally coincides with the epigraph of $\varphi$. We can also assume that
\[\frac{\partial\varphi}{\partial x_k}(0',0)= \frac{\partial\varphi}{\partial y_j}(0',0)=0\] for $k\in 1,\cdots,n-1$ and $j=1,\cdots,n$. Then, the real vector field
\[T=\frac{\partial}{\partial x_k}+(\varphi_{x_k})(z',y_n)\frac{\partial}{\partial x_n}\] is tangential to $\partial\Omega$.

Hence, taking the barrier
\[ v(z)=T(u-h_\varphi-a\rho)(z)+\big(u_{y_n}(z)-\rho_{y_n}(z)\big)^2+a\big(\rho(z)-u(z)\big)-c||z||^2,\]
with $\rho$ as in (\ref{syli}), for the suitably chosen $b,c>0$ one has $v(z)\leq 0$ for $z\in\Omega$ near zero. This yields one sided bound on $\frac{\partial^2u}{\partial x_kx_n}(0)$. Applying the same barrier with $T$ replaced by $-T$ gives control in the other direction.

In order to bound $\frac{\partial^2u}{\partial y_jx_n}(0)$ one has to replace $T$ by $\pm [\frac{\partial}{\partial y_j}+(\varphi_{y_j})(z',y_n)\frac{\partial}{\partial x_n}]$.

\end{proof}

Henceforth, we shall assume that all the above quantities are bounded by a generic constant $C$.

\section{Weak interior estimates}\label{wint}

In this section, we state our result for a complex Hessian equation associated with an operator $\mathcal F$ and an associated cone $\Gamma$. We shall follow the real approach of \cite{Iv}.  

For this goal, let in this section $\Omega$ be a strictly $\Gamma$-pseudoconvex domain with $C^4$ boundary and $\psi$ be a defining function for $\Omega$ as in Definition \ref{pseudoconv}. Modifying $\psi$, if necessary, we can assume that $\psi$ is strictly $\Gamma$-admissible, recall the Definition \ref{admis}. 
	
	In $\Om$ we consider a linear elliptic operator
	\[ \mathcal L[u]=\sum_{i,j=1}^n\aij u_{z_i\bar{z}_j}=tr[(\aij)D^2_{\mathbb C}u^{T}]\]
for the coefficients $\aij=\aij(z)$ we shall assume that
\begin{equation}\label{tra}
	\forall_{z \in \overline{\Omega}} \ tr[\aij]=\sum_{j=1}^na^{i\bar{i}}=1,
\end{equation} 
and
\begin{equation}\label{dualcone}
\forall_{z\in\overline{\Omega}}\ \aij\in\tilde{\Gamma}^*,	
\end{equation}
where $\tilde{\Gamma}^*\subset\tilde{\Gamma}_n$ is the dual cone to $\tilde{\Gamma}$ defined by 
$$\tilde{\Gamma}^*=\lbrace A=(\aij)|\ \forall_{B=(b_{ij})\in\tilde{\Gamma}}\ tr(AB)>0\rbrace.$$
	Note that in particular this implies that the largest eigenvalue of $(\aij)$ is at most $1$. The link between the operator and the domain is through the following fundamental assumption
	\begin{equation}\label{apsi}
		\mathcal L[\psi]\leq -2. 
	\end{equation}
Note that since $-\psi$ is strictly $\Gamma$-admissible up to the boundary $\mathcal L[\psi]$ is bounded from above by a negative constant which can always be assumed to be equal to $-2$ by rescaling.

\subsection{Interior gradient estimate}\label{gradint}
We aim for a complex analytic counterpart of the gradient bound in \cite{Iv}. We note that a similar gradient bound was obtained in Theorem 3.1 from \cite{Zh}, due to Zhou, who used probabilistic arguments. Such weakly interior gradient estimates were first obtained by Krylov, see Example 5.3 from \cite{Kr93} - there the bound is obtained for the complex Monge-Amp\`ere equation exclusively but the argument generalizes for general Hessian equations satisfying (\ref{below}). Although the result in this subsection is not strictly needed for the main theorem, we included it to highlight the strength of Krylov's approach. {More importantly, we need most of the computations for the subsequent subsection concerning second order bounds.}
	\begin{theorem}\label{1Zhou}
	Let $\Om$ and $\psi$ be as above. Let the $\Gamma$-admissible function $u$, continuous up to the boundary of $\Omega$ solve the equation
	\begin{equation}
	\begin{cases}\mathcal{F}(u_{z_i\bar{z}_j})=f,\\
		u|_{\partial \Om } = g.
	\end{cases} 	\end{equation}
 Assume that $f, g\in C^1(\overline{\Om})$ and, naturally, $f\geq 0$ in $\Omega$. Then for any $\epsilon>0$ there is a constant $C_\epsilon=C\left(||f||_{C^1}, {||g||_{C^1}}, ||\psi||_{C^3}, n, \epsilon \right)$, but independent of $\inf f$, such that for any unit length vector $\zeta\in \mathbb C^n$ one has
\begin{equation}
|u_{(\zeta)}|\leq \frac{\epsilon}{\psi^{\frac{1}{2}}}  \sup_{\partial \Omega} |u_\nu| + C_\epsilon,\end{equation}
where $\nu$ is the inner normal along $\partial \Omega$.
	\end{theorem}
	
	The equation above may be understood in the viscosity \cite{CIL92} or pluripotential \cite{BT76} sense (whenever pluripotential theory associated to $F$ can be constructed).
	
	Fix $\alpha\in(0,1),\ \beta\in(0,1)$, $\beta<<\alpha$.\footnote{$\al$ will depend on the geometry of $\Omega$ - in particular on the $C^3$ norm of $\psi$.} We consider the operator
	\begin{equation}\label{L}
		L[w]:=\sum_{i,j}A^{i\bar{j}}w_{i\bar{j}}-2\Re \left(\sum_kb_kw_k \right),
	\end{equation}
	 with, {here,} $i, j$ and $k$ running over the $2n+1=n+n+1$ complex variables $z_1,\cdots,z_n,\zeta_1,\cdots,\zeta_n,\zeta_0$. We denote by $\zeta=(\zeta_1,\cdots,\zeta_n)$ and by $\zeta'=(\zeta_1,\cdots,\zeta_n, \zeta_0)$\footnote{The somewhat unusual enumeration follows \cite{Iv}.}. In this section our notation for multiple derivatives simplifies as we use only constant vector fields.

The matrix $A^{i\bar{j}}$ is given symbolically by the formula (compare \cite{Iv}, \cite{ITW}):
\begin{equation}\label{AIJ}
	\left( A^{i\bar{j}} \right):=\begin{pmatrix}
		a^{i\bar{j}} & r a^{i\bar{j}} & a^{i\bar{j}}\overline{q}_j\\
		\overline{r} \overline{a^{j\bar{i}}}&|r|^2 a^{i\bar{j}}&ra^{i\bar{j}}\overline{q}_j\\
		q_i\overline{a^{j\bar{i}}}&\overline{r}q_i\overline{a^{j\bar{i}}}&q_ia^{i\bar{j}}\overline{q}_j
	\end{pmatrix} 
	=\begin{pmatrix}
		A &rA&A\bar{q}\\
		\bar{r}A&|r|^2A&\bar{r}A\bar{q}\\
		q^T A&{r}q^T A&q^T A \bar{q}
	\end{pmatrix},	
\end{equation}
with the following choices made:

\begin{equation}\label{aij}
	a^{i\bar{j}}:=\frac{\frac{\partial \mathcal{F}\left(D_\cc^2u\right)}{\partial u_{i\bar{j}}}}{\sum_{l=1}^n\frac{\partial \mathcal{F}\left(D_\cc^2u\right)}{\partial u_{l\bar{l}}}};
	\end{equation}
\begin{equation}\label{r}
	r(z):=\frac{\alpha\psi_{\overline{\zeta}}}{\psi+\beta};
\end{equation}
\begin{equation}\label{b}
	b_i=\begin{cases}
		0\ {\rm if}\ i\leq n\ {\rm {or}}\ i>2n\\
		a^{(i-n)\overline{j}}\overline{q}_j
	\end{cases};
	\end{equation}
\begin{equation}\label{q}
	q_j:=\frac{1}{\psi+\beta}\large(\frac{\overline{\zeta_j}}{4}+\alpha\frac{\psi_{j}\psi_{\overline{\zeta}}}{\psi+\beta}
	\large).
\end{equation}

We assume that $f$ is uniformly bounded and that $||f||_{C^1}$ is also under control. The first lemma we prove establishes a lower bound for a perturbed first order derivative of $u$. In the following proof, abusing slightly the notation, we shall use $|\zeta|$ for the norm of a vector $\zeta\in \cc^n$, leaving $||\xi||$ for vectors living in $\mathbb C^{n+1}$.

\begin{lemma}\label{perturbation}
	Let $w(z,\zeta')={u_{(\zeta)}}+2\Re(\zeta_0)u(z)$. Then
	$$L[w]\geq -\mu \left(|\zeta|+\beta^{-1}|\zeta|+|\zeta_0|\right),$$
	with a constant $\mu$ depending only on $\|f\|_{C^1}$, $\|\psi\|_{C^1}$ and $\alpha$.
\end{lemma} 
\begin{proof}
		By direct computations we obtain
\[w_k=u_{k-n}, \: n<k<2n+1,\]
	\[\begin{pmatrix}
	(w_{i\bar{j}})& (w_{i\overline{j+n}})& (w_{i\overline{2n+1}})\\
		(w_{i+n\bar{j}})&(w_{i+n\overline{j+n}})&(w_{i+n\overline{2n+1}})\\
		(w_{2n+1\bar{j}})&(w_{2n+1\overline{j+n}})&(w_{2n+1\overline{2n+1}})
	\end{pmatrix}=\begin{pmatrix}
		u_{\zeta i \bar{j}}+u_{\bar{\zeta}i\bar{j}}+\zeta_0 u_{i\bar{j}}+\bar{\zeta_0} u_{i\bar{j}} & u_{i\bar{j}}& u_i\\
		u_{i\bar{j}}&0&0\\
		u_{\bar{j}}&0&0
	\end{pmatrix}
	.\]
	Consequently:
	\begin{align*}
		L[w]=\aij(u_{\zeta i\bar{j}}+u_{\bar{\zeta}i\bar{j}})+2\Re\large((r+\zeta_0)\aij u_{i\bar{j}}\large)\\
		+2\Re(\aij\overline{q}_ju_i)-2\Re(b_ju_j).
	\end{align*}
	The terms on the last line cancel each other from the very definition of $b_j$ . The first line becomes
	\begin{equation}\label{Lw}
	L[w]=\frac{2\Re(f_{\zeta})}{\sum_{l=1}^n\frac{\partial \mathcal{F}\left(D_\cc^2u\right)}{\partial u_{l\bar{l}}}}+\frac{2\Re(r+\zeta_0)f}{\sum_{l=1}^n\frac{\partial \mathcal{F}\left(D_\cc^2u\right)}{\partial u_{l\bar{l}}}}\geq -\mu \left(|\zeta|+\beta^{-1}|\zeta|+|\zeta_0|\right),
	\end{equation}
	where we used the fact that $\sum_{l=1}^n\frac{\partial \mathcal{F}\left(D_\cc^2u\right)}{\partial u_{l\bar{l}}}$ is bounded from below. This finishes the proof of the lemma.
\end{proof}

Recall that we make the standing assumption that
\begin{equation}\label{psi}
	\aij\psi_{i\bar{j}}\leq -2.
\end{equation}
For a barrier we shall use morally the same function as in \cite{Iv}. Take 
$$v_1:=\frac{|\psi_\zeta|^2}{(\psi+\beta)^{\alpha}},\ v_2:=\frac1{\alpha}(\psi+\beta)^{1-\alpha}|\zeta|^2,\ v_3:=(\psi+\beta)^{1-\alpha}|\zeta_0|^2.$$

The barrier function $v$ will be 
\begin{equation} \label{barr} v:= \sqrt{v_1+v_2+\frac\beta4 v_3}. \end{equation} 
Note that, just as in \cite{Iv}, we have
\begin{equation}\label{Lv}
	L[v]=\frac1{2v}L[v_1+v_2+\frac\beta4 v_3]-A^{i\bar{j}}\frac{[v_1+v_2+\frac{\beta}{4} v_3]_i[v_1+v_2+\frac\beta4 v_3]_{\bar{j}}}{4v^3}
\end{equation}
\[\leq \frac1{2v}L[v_1+v_2+\frac\beta4 v_3].\]

\begin{lemma}\label{barrier}
We have for some constant $c(n)$, dependent only on $n$, that:
$$L\left[v_1+v_2+\frac\beta4 v_3\right]\leq -\frac{1}{c(n)(\psi+\beta)^{\al}}\left[\frac1\al|\zeta|^2+\al\frac{|\psi_\zeta|^2}{\psi+\beta}+\beta|\zeta_0|^2\right].$$	
\end{lemma} 
\begin{proof} It suffices to estimate $L$ on each of the $v_i$'s.

The computation will be separated into several steps.

{\bf Step 1. Computation of $L[v_1]$}

By direct computation we have
\begin{align*}
L[v_1]=\frac{\aij}{(\psi+\beta)^{\alpha}}[\psi_{\zeta i\bar{j}}\psi_{\bar{\zeta}}+\psi_{\bar{\zeta}i\bar{j}}\psi_{\zeta}+\psi_{\zeta i}\psi_{\bar{\zeta}\bar{j}}+\psi_{\zeta\bar{j}}\psi_{\bar{\zeta}j}]\\
-\al\frac{\aij}{(\psi+\beta)^{\alpha+1}}[\psi_{\bar{j}}(\psi_{\zeta i}\psi_{\bar{\zeta}}+\psi_{\bar{\zeta}i}\psi_{\zeta})+\psi_{i}(\psi_{\zeta \bar{j}}\psi_{\bar{\zeta}}+\psi_{\bar{\zeta}\bar{j}}\psi_{\zeta})]\\
	-\al\frac{\aij}{(\psi+\beta)^{\alpha+1}}|\psi_{\zeta}|^2\psi_{i\bar{j}}+\al(\al+1)\frac{\aij}{(\psi+\beta)^{\alpha+2}}\psi_i\psi_{\bar{j}}|\psi_{\zeta}|^2\\
	+2\Re[r\frac{\aij}{(\psi+\beta)^{\alpha}}(\psi_{i\bar{j}}\psi_{\zeta}+\psi_{\bar{j}}\psi_{\zeta i}-\al\frac{\psi_{\bar{j}}\psi_{\zeta}\psi_i}{\psi+\beta})]\\
	+|r|^2\frac{\aij}{(\psi+\beta)^{\alpha}}\psi_i\psi_{\bar{j}}-2\Re[\frac{\aij}{(\psi+\beta)^{\alpha}}\overline{q}_{j}\psi_i\psi_{\bar{\zeta}} ]\\
	=:A+B+C+D+E+F+G.
\end{align*}
Formulas (\ref{r}) and (\ref{q}) for $r$ and $q_j$ imply that 
\begin{align*}
E+F+G=2\Re[\frac{\aij}{(\psi+\beta)^{\alpha}}(\frac{\al|\psi_{\zeta}|^2\psi_{i\bar{j}}+
\al\psi_{\bar{j}}\psi_{\zeta i}\psi_{\bar{\zeta}}}{\psi+\beta}-\al^2\frac{|\psi_{\zeta}|^2\psi_i\psi_{\bar{j}}}{(\psi+\beta)^2})]\\
+\al^2\frac{a^{i\bar{j}}|\psi_{\zeta}|^2\psi_i\psi_{\bar{j}}}{(\psi+\beta)^{2+\alpha}}-\frac12\Re[\frac{\aij}{(\psi+\beta)^{\alpha+1}}(\zeta_j\psi_i\psi_{\bar{\zeta}})]-2\al\Re[\frac{\aij}{(\psi+\beta)^{\alpha+2}}|\psi_{\zeta}|^2\psi_i\psi_{\bar{j}}].
\end{align*}
Coupling the above with $A+B+C+D$ and simplifying one obtains

\begin{align*}
L[v_1]=\frac{\aij}{(\psi+\beta)^{\alpha}}[\psi_{\zeta i\bar{j}}\psi_{\bar{\zeta}}+\psi_{\bar{\zeta}i\bar{j}}\psi_{\zeta}+\psi_{\zeta i}\psi_{\bar{\zeta}\bar{j}}+\psi_{\zeta\bar{j}}\psi_{\bar{\zeta}j}]\\
-\al\frac{\aij}{(\psi+\beta)^{\alpha+1}}[\psi_{\bar{j}}\psi_{\bar{\zeta}i}\psi_{\zeta}+\psi_{i}\psi_{\zeta \bar{j}}\psi_{\bar{\zeta}}]
+\al\frac{\aij}{(\psi+\beta)^{\alpha+1}}|\psi_{\zeta}|^2\psi_{i\bar{j}}\\
-\al\frac{\aij}{(\psi+\beta)^{\alpha+2}}|\psi_{\zeta}|^2\psi_i\psi_{\bar{j}}-\frac12\Re[\frac{\aij}{(\psi+\beta)^{\alpha+1}}(\zeta_j\psi_i\psi_{\bar{\zeta}})].
	\end{align*}

Below we follow \cite{Iv} and by $\Psi_j$, $j=1,2,3$ we denote the semi-norms $||\psi||_j$, respectively. Using additionally (\ref{psi}) we can further estimate

\begin{align*}\begin{gathered}
	L[v_1]\leq \frac{\aij}{(\psi+\beta)^{\alpha}}\left(2\Psi_3\Psi_1+2\Psi_2^2
    {-\frac{\al}{(\psi+\beta)} \left[\psi_{\bar{j}}\psi_{\frac{\bar{\zeta}}{|\zeta|}i}\psi_{\frac{\zeta}{|\zeta|}}+\psi_{i}\psi_{\frac{\zeta}{|\zeta|} \bar{j}}\psi_{\frac{\bar{\zeta}}{|\zeta|}} \right]}
    \right)|\zeta|^2\\-2\al\frac{1}{(\psi+\beta)^{\alpha+1}}|\psi_{\zeta}|^2 
	-\al\frac{\aij}{(\psi+\beta)^{\alpha+2}}|\psi_{\zeta}|^2\psi_i\psi_{\bar{j}}-\frac12\Re[\frac{\aij}{(\psi+\beta)^{\alpha+1}}(\zeta_j\psi_i\psi_{\bar{\zeta}})]\\=A_1+B_1+C_1+D_1.\end{gathered}
	\end{align*}

{\bf Step 2. Computation of $L[v_2]$}

We have
\begin{align*}
	L[v_2]=\frac{\aij}{(\psi+\beta)^{\alpha}}\left[\frac{1-\al}{\al}\psi_{i\bar{j}}|\zeta|^2-(1-\al)\frac{\psi_i\psi_{\bar{j}}|\zeta|^2}{\psi+\beta}\right]\\
	+2\frac{(1-\al)}{\al}\Re\left[r\frac{\aij}{(\psi+\beta)^{\alpha}}\psi_i\zeta_j\right]+\frac{|r|^2}{\al}\frac{\aij}{(\psi+\beta)^{\alpha-1}}\delta_{ij}\\-\frac2{\al}\Re\left[\frac{\aij}{(\psi+\beta)^{\alpha-1}}\overline{q}_j\bar{\zeta}_i\right]=
	I+II+III+IV.
\end{align*}
Next we exploit the formulas for $r$ and $q$. The sum $II+III+IV$ becomes
\begin{equation*}
	-2\al\Re\left[\frac{\aij}{(\psi+\beta)^{\alpha+1}}(\zeta_j\psi_i\psi_{\bar{\zeta}})\right]+\al\frac{\aij}{(\psi+\beta)^{\alpha+1}}|\psi_{\zeta}|^2\delta_{ij}
	-\frac{1}{2\al}\frac{\aij}{(\psi+\beta)^{\alpha}}\zeta_j\bar{\zeta}_i,
	\end{equation*}
and hence
\begin{align*}
	L[v_2]\leq -2\frac{1-\al}{\al}\frac{1}{(\psi+\beta)^{\al}}|\zeta|^2-(1-\al)\frac{\aij}{(\psi+\beta)^{\alpha+1}}\psi_i\psi_{\bar{j}}|\zeta|^2\\
	-2\al\Re\left[\frac{\aij}{(\psi+\beta)^{\alpha+1}}(\zeta_j\psi_i\psi_{\bar{\zeta}})\right]+\al\frac{1}{(\psi+\beta)^{\alpha+1}}|\psi_{\zeta}|^2
	-\frac{1}{2\al}\frac{\aij}{(\psi+\beta)^{\alpha}}\zeta_j\bar{\zeta}_i\\
	=I_1+II_1+III_1+IV_1+V_1
\end{align*}
{\bf Step 3. Estimation of $L[v_1+v_2]$}
Before we proceed, the following observations are due. Note that for $\alpha$ sufficiently close to zero (dependent on $\Psi_j$, $j=1,2,3$) the first two summands of $A_1$ are absorbed by {half} of $I_1$. Next, half of $B_1$ is canceled with $IV_1$. Finally, $D_1$ agrees with $III_1$ up to a multiplicative constant. Hence, for $\al$ small enough:
\begin{align*}
	L[v_1+v_2]\leq -2\frac{1-\al}{2\al}\frac{1}{(\psi+\beta)^{\al}}|\zeta|^2-{(1-\al)}\frac{\aij}{(\psi+\beta)^{\alpha+1}}\psi_i\psi_{\bar{j}}|\zeta|^2
    \\ {-\frac{2\al}{(\psi+\beta)^{\al+1}} \Re\left[\aij(\psi_{\bar{j}}|\zeta|)(\psi_{\frac{\bar{\zeta}}{|\zeta|}i}\psi_{\zeta})\right]}
    -\al\frac{1}{(\psi+\beta)^{\alpha+1}}|\psi_{\zeta}|^2\\
	-\al\frac{\aij}{(\psi+\beta)^{\alpha+2}}|\psi_{\zeta}|^2\psi_i\psi_{\bar{j}}-(\frac12+2\al)\Re[\frac{\aij}{(\psi+\beta)^{\alpha+1}}(\zeta_j\psi_i\psi_{\bar{\zeta}})]
	-\frac{1}{2\al}\frac{\aij}{(\psi+\beta)^{\alpha}}\zeta_j\bar{\zeta}_i.
\end{align*}
Note that, provided $\al$ is small enough and since the standard hermitian form dominates $(a^{i\bar{j}})$:
\[ \begin{gathered} 
-{2\al}\frac{\aij}{(\psi+\beta)^{\alpha+1}}\psi_i\psi_{\bar{j}}|\zeta|^2-\frac\al2\frac{1}{(\psi+\beta)^{\alpha+1}}|\psi_{\zeta}|^2-2\al\Re[\frac{\aij}{(\psi+\beta)^{\alpha+1}}(\zeta_j\psi_i\psi_{\bar{\zeta}})]\\ \leq 
 -2 \al \frac{\aij}{(\psi+\beta)^{\alpha+1}}\psi_i\psi_{\bar{j}}|\zeta|^2-\frac\al2\frac{1}{(\psi+\beta)^{\alpha+1}} \left|\frac{\psi_{\zeta}}{|\zeta|} \cdot \bar{\zeta} \right|^2-2\al\Re[\frac{\aij}{(\psi+\beta)^{\alpha+1}}(\zeta_j\psi_i\psi_{\bar{\zeta}})]\\
\leq 
-2 \al \frac{\aij}{(\psi+\beta)^{\alpha+1}}\psi_i\psi_{\bar{j}}|\zeta|^2
-\frac\al2\frac{a^{i\bar{j}}}{(\psi+\beta)^{\alpha+1}}\frac{\psi_{\zeta}}{|\zeta|} \bar{\zeta_i}\frac{\psi_{\bar{\zeta}}}{|\zeta|} \zeta_j -2\al\Re[\frac{\aij}{(\psi+\beta)^{\alpha+1}}(\zeta_j\psi_i\psi_{\bar{\zeta}})]\\
\leq 0.
\end{gathered} \]
{We remark that if $|\zeta|=0$ the above inequality still holds. Moreover:}
$$-\frac{\al}2\frac{\aij}{(\psi+\beta)^{\alpha+2}}|\psi_{\zeta}|^2\psi_i\psi_{\bar{j}}-\frac12\Re[\frac{\aij}{(\psi+\beta)^{\alpha+1}}(\zeta_j\psi_i\psi_{\bar{\zeta}})]
-\frac{1}{4\al}\frac{\aij}{(\psi+\beta)^{\alpha}}\zeta_j\bar{\zeta}_i\leq 0.$$

Finally
{\begin{align*}
    \begin{gathered}
        -\frac{2\al}{(\psi+\beta)^{1+\al}} \Re\left[\aij(\psi_{\bar{j}}|\zeta|)(\psi_{\frac{\bar{\zeta}}{|\zeta|}i}\psi_{\zeta})\right] \\ \leq \frac{1-3\al}{(\psi+\beta)^{1+\al}}\aij \psi_i\psi_{\bar{j}}|\zeta|^2+ \frac{\al^2}{(1-3\al)(\psi+\beta)^{1+\al}}\aij |\psi_\zeta|^2\psi_{\frac{\bar{\zeta}}{|\zeta|}i}\psi_{\frac{\zeta}{|\zeta|}\bar{j}}
        \\ \leq \frac{1-3\al}{(\psi+\beta)^{1+\al}}\aij \psi_i\psi_{\bar{j}}|\zeta|^2+ 
        \frac{\al^2}{(1-3\al)(\psi+\beta)^{1+\al}} |\psi_\zeta|^2 \Psi_2^2
        \\ \leq \frac{1-3\al}{(\psi+\beta)^{1+\al}}\aij \psi_i\psi_{\bar{j}}|\zeta|^2+ 
        \frac{\al}{4(\psi+\beta)^{1+\al}} |\psi_\zeta|^2,
    \end{gathered}
\end{align*}}
{where the last inequality holds for $\al$ sufficiently small in comparison to $\Psi_2$ while the penultimate one follows from the bound on the maximal eigenvalue of $(\aij)$.}

{Utilizing the last three inequalities} we obtain:
\begin{equation} \begin{split}\label{Lv12}
	L[v_1+v_2]\leq -{\frac{1-\al}{\al}}\frac{1}{(\psi+\beta)^{\al}}|\zeta|^2
	-{\frac{\al}4}\frac{1}{(\psi+\beta)^{\alpha+1}}|\psi_{\zeta}|^2 
	\\
	-\frac{\al}2\frac{\aij}{(\psi+\beta)^{\alpha+2}}|\psi_{\zeta}|^2\psi_i\psi_{\bar{j}}
-\frac{1}{4\al}\frac{\aij}{(\psi+\beta)^{\alpha}}\zeta_j\bar{\zeta}_i.\end{split}
\end{equation}
	{\bf Step 4. Computation of $L[v_3]$}
	
	By brute-force calculation one has
	\begin{align*}
		L[v_3]=(1-\al)\frac{\aij}{(\psi+\beta)^{\alpha}}\psi_{i\bar{j}}|\zeta_0|^2-\al(1-\al)\frac{\aij}{(\psi+\beta)^{\alpha+1}}\psi_i\psi_{\bar{j}} |\zeta_0|^2\\
		+2(1-\al)\Re\left[\frac{\aij}{(\psi+\beta)^{\alpha}}\overline{q}_j\psi_i\zeta_0\right]+q_i\frac{\aij}{(\psi+\beta)^{\alpha-1}}\overline{q}_j.
	\end{align*}
	Once again the definition of $q_i$, together with assumption (\ref{psi}) lead to
	\begin{align*}
	L[v_3]\leq -2\frac{|\zeta_0|^2}{(\psi+\beta)^{\alpha}}-\al(1-\al)\frac{\aij}{(\psi+\beta)^{\alpha+1}}\psi_i\psi_{\bar{j}} |\zeta_0|^2\\
	+\frac{(1-\al)}{2}\Re\left[\frac{\aij}{(\psi+\beta)^{\alpha+1}}\psi_i\zeta_j\zeta_0\right]+2(1-\al)\al\Re\left[\frac{\aij}{(\psi+\beta)^{\alpha+2}}\psi_i\psi_{\bar{j}}\psi_\zeta\zeta_0\right]\\
	+\frac1{16}\frac{\aij}{(\psi+\beta)^{\alpha+1}}\overline{\zeta_i}\zeta_j+\frac{\al}2\Re\left[\frac{\aij}{(\psi+\beta)^{\alpha+2}}\overline{\zeta_i}\psi_{\bar{j}}\psi_\zeta\right]\\
	+\al^2\frac{\aij}{(\psi+\beta)^{\alpha+3}}|\psi_{\zeta}|^2\psi_i\psi_{\bar{j}}=(1)+(2)+\cdots+(7).
	\end{align*} 
	Observe now that
	\begin{equation}\label{3}
	(2)+(3)\leq -\frac{(1-\al)\al}{2}\frac{\aij}{(\psi+\beta)^{\alpha+1}}\psi_i\psi_{\bar{j}}|\zeta_0|^2+\frac{(1-\al)}{8\al}\frac{\aij}{(\psi+\beta)^{\alpha+1}}\overline{\zeta_i}\zeta_j,
	\end{equation}
	 while
	\begin{equation}\label{4}
		(4)\leq \frac{(1-\al)\al}{2}\frac{\aij}{(\psi+\beta)^{\alpha+1}}\psi_i\psi_{\bar{j}}|\zeta_0|^2+\frac{2(1-\al)\al^2}{\al}\frac{\aij}{(\psi+\beta)^{\alpha+3}}\psi_i\psi_{\bar{j}}|\psi_\zeta|^2.		
	\end{equation}
	Finally
	\begin{equation}\label{6}
	(6)\leq \frac1{16}\frac{\aij}{(\psi+\beta)^{\alpha+1}}\overline{\zeta_i}\zeta_j+\al^2\frac{\aij}{(\psi+\beta)^{\alpha+3}}\psi_i\psi_{\bar{j}}|\psi_\zeta|^2
	\end{equation}
	$$=(5)+(7).$$
	
	Coupling now the estimation of $L[v_3]$ with (\ref{3}), (\ref{4}) and (\ref{6}) we obtain
	
	\begin{equation}\label{v3semifinal}
L[v_3]\leq -2\frac{|\zeta_0|^2}{(\psi+\beta)^{\alpha}}+\frac1{8\al}\frac{\aij}{(\psi+\beta)^{\alpha+1}}\overline{\zeta_i}\zeta_j+2\al\frac{\aij}{(\psi+\beta)^{\alpha+3}}\psi_i\psi_{\bar{j}}|\psi_\zeta|^2.
	\end{equation}
	
	Plugging $\frac{\beta}{(\psi+\beta)}\leq1$ into (\ref{v3semifinal}) we finally finish with
	\begin{equation}\label{v3final}
	L\left[\frac{\beta}4v_3\right]\leq -2\frac{\beta}4\frac{|\zeta_0|^2}{(\psi+\beta)^{\alpha}}+\frac1{32\al}\frac{\aij}{(\psi+\beta)^{\alpha}}\overline{\zeta_i}\zeta_j+\frac\al2\frac{\aij}{(\psi+\beta)^{\alpha+2}}\psi_i\psi_{\bar{j}}|\psi_\zeta|^2.
	\end{equation}
	{\bf Step 5. Estimation of $L\left[v_1+v_2+\frac{\beta}4v_3\right]$}
	
	Coupling (\ref{Lv12}) with (\ref{v3final}) we obtain
	
	\begin{equation}\label{123}\begin{gathered}
	L[v_1+v_2+\frac{\beta}4v_3]\\
    \leq -{\frac{1-\al}{\al}}\frac{1}{(\psi+\beta)^{\al}}|\zeta|^2-{\frac{\al}4}\frac{1}{(\psi+\beta)^{\alpha+1}}|\psi_{\zeta}|^2-{\frac{\beta}{2}}\frac{|\zeta_0|^2}{(\psi+\beta)^{\alpha}}.\end{gathered}
	\end{equation}
This finishes the proof of Lemma \ref{barrier}.
	\end{proof}
We now proceed to the proof of the main result of this subsection.
\begin{proof}[Proof of Theorem \ref{1Zhou}]
Using Krylov's maximum principle, see Theorem \ref{phomog}, we know that:
\begin{equation} \begin{gathered} \label{krylovmax}	
	\frac{w}{v}\leq \max \Big\lbrace \sup\limits_{\Omega^N} \frac{L[w]}{L[v]}, \sup\limits_{\partial \Omega^N}\frac wv \Big\rbrace\\
	\leq 	\max \Big\lbrace \sup\limits_{\Omega^N} \frac{-L[w]}{-L[v]}, \sup\limits_{\partial \Omega^N}\frac wv \Big\rbrace.
\end{gathered}\end{equation}
For the first term by Lemma \ref{perturbation} and (\ref{Lv}) we have:
\[\frac{-L[w]}{-L[v]} \leq 2 \cdot \frac{\mu \left(|\zeta|+\beta^{-1}|\zeta|+|\zeta_0|\right)}{v} \cdot \frac{v_1+v_2+\frac\beta4 v_3}{-L[v_1+v_2+\frac\beta4 v_3]}.\]
Using Lemma \ref{barrier} and further modifying $\mu$ if necessary, depending on $\alpha$, we can estimate the second fraction by:
\[\frac{-L[w]}{-L[v]} \leq (\psi + \beta) \cdot \frac{\mu \left(|\zeta|+\beta^{-1}|\zeta|+|\zeta_0|\right)}{v}.\]
Referring to the very definition (\ref{barr}) of $v$ we further have
\begin{equation} \label{firstterm} \frac{-L[w]}{-L[v]} \leq \big(\sup\limits_{\overline{\Omega}} \psi + \beta\big)^{\frac{1+\alpha}{2}} \mu \left(\beta^{-\frac 1 2}+\beta^{-\frac 3 2}\right):= C_\beta.\end{equation}
As for the second term in (\ref{krylovmax}), since both $v$ and $w$ are homogeneous in the $\xi'$ variable we assume $z \in \partial \Omega$ and $|(\zeta,\zeta_0)|=1$ are such that
\begin{equation} \label{second} \sup\limits_{\partial \Omega^N}\frac wv = \frac{w}{v}(z,\zeta'). \end{equation}
Let $\nu$ be the unit inner normal vector to $\partial\Omega$ and let 
\begin{equation} \big|(\Re (\zeta), \nu)\big|=\theta.\end{equation}

We have, {due to estimates from Section 3 and properties of $\psi$ from Definition \ref{pseudoconv},} at $(z,\zeta')$:
\begin{equation}
\begin{gathered}
w(z,\zeta') \leq {2}\theta \sup_{\partial \Omega} |u_\nu| + C_\beta,\\
v(z,\zeta') \geq {\frac 1 2}\theta \beta^{\frac{-\alpha}{2}}+ C_\beta.
\end{gathered}
\end{equation}
We remark that in the first inequality the term $C_{\beta}$ absorbs the tangential derivatives of $u$ (which are equal to the corresponding tangential derivatives of the boundary data $g$).

All this results {in a sequence of inequalities:}
\begin{equation}\label{secondterma}
\sup\limits_{\partial \Omega^N}\frac wv {\leq \frac{2\theta \sup_{\partial \Omega} |u_\nu|}{v} + \frac{C_\beta}{v} \leq 4 \beta^{\frac \alpha 2}\sup_{\partial \Omega} |u_\nu| + C_\beta} \leq \epsilon \sup_{\partial \Omega} |u_\nu| + C_\epsilon
\end{equation}
for $\epsilon>0$ as small as we wish at the cost of making $\beta$ sufficiently small and, potentially, $C_\epsilon$ large.
Applying (\ref{firstterm}) and (\ref{secondterma}) in (\ref{krylovmax}) for $\zeta_0=0$ and any norm one $\zeta$ we get for all $z \in \Omega$:
\begin{equation}
u_{(\zeta)}(z) \leq v(z,\zeta) \Big( \epsilon \sup_{\partial \Omega} |u_\nu| + C_\epsilon \Big)\leq \frac{\epsilon}{\psi^{\frac{\alpha}{2}}}  \sup_{\partial \Omega} |u_\nu| + C_\epsilon \leq 
\frac{\epsilon}{\psi^{\frac{1}{2}}}  \sup_{\partial \Omega} |u_\nu| + C_\epsilon, \end{equation}
as $\alpha<1$.
\end{proof}

\subsection{Interior hessian estimate}\label{hessint}
In this section again we use only constant vector field $\zeta \in \mathbb{C}^n$ thus we recall that the first order operator is given by:
\[ \label{crazyoper} u_{(\zeta)}:= u_\zeta + u_{\bar{\zeta}}= \zeta_i u_i + \bar{\zeta_i} u_{\bar{i}}.\]
One can easily see that we also have the formulas:
\[u_{(\zeta)(\zeta)}:= u_{\zeta \zeta} + 2 u_{\zeta \bar{\zeta}}+ u_{\bar{\zeta}\bar{\zeta}}= \zeta_i\zeta_j u_{ij} + 2\zeta_i\bar{\zeta_j} u_{i\bar{j}} + \bar{\zeta_i} \bar{\zeta_j} u_{\bar{i}\bar{j}}.\]
The notation we shall use is as in the previous subsection. This time we prove the following lemma:
\begin{lemma}\label{Lu}
Let $w(z,\zeta')=u_{(\zeta) (\zeta)}+ 4\Re(\zeta_0) u_{(\zeta)}+|2\Re(\zeta_0)|^2u(z)$. Then
	$$L[w]\geq -\mu||\zeta'||^2,$$
	with a constant $\mu$ depending only on $||f||_{C^2}$, $n$ and $\Omega$.
\end{lemma} 
\begin{proof}
	Let us rewrite:
	\[w(z,\zeta')=(\zeta_i\zeta_j u_{ij}+2\zeta_i \bar{\zeta_j}u_{i\bar{j}}+\bar{\zeta_i}\bar{\zeta_j}u_{\bar{i}\bar{j}})+ 2(\zeta_0+\bar{\zeta_0})u_{(\zeta)}+(\zeta_0+\bar{\zeta_0})^2u.\]
	Then, by a direct computations:
	\[w_{n+k}=2u_{k\bar{\zeta}}+2u_{\zeta k}+4 \Re(\zeta_0) u_{k}=2u_{(\zeta)k}+4 \Re(\zeta_0) u_k, \: 1\leq k\leq n,\]
	\[\begin{gathered}\begin{pmatrix}
			(w_{i\bar{j}})& (w_{i\overline{j+n}})& (w_{i\overline{2n+1}})\\
			(w_{i+n\bar{j}})&(w_{i+n\overline{j+n}})&(w_{i+n\overline{2n+1}})\\
			(w_{2n+1\bar{j}})&(w_{2n+1\overline{j+n}})&(w_{2n+1\overline{2n+1}})
		\end{pmatrix}
	=\begin{pmatrix}
		u_{(\zeta) (\zeta) i \bar{j}} & 2 u_{(\zeta) i\bar{j}} & 0\\
		2u_{(\zeta)i \bar{j}}&2u_{i\bar{j}}&0\\
		0&0&0
	\end{pmatrix}\\
	+\begin{pmatrix}
		4 \Re(\zeta_0) u_{(\zeta) i \bar{j}} & 4\Re(\zeta_0) u_{i\bar{j}}& 2u_{(\zeta) i}\\
		4\Re(\zeta_0) u_{i \bar{j}}&0&2u_i\\
		2 u_{(\zeta) \bar{j}}&2u_{\bar{j}}&0
	\end{pmatrix} 
	+\begin{pmatrix}
		|2\Re(\zeta_0)|^2u_{i\bar{j}}&0&4\Re(\zeta_0)u_i\\
		0&0&0\\
		4\Re(\zeta_0)u_{\bar{j}}&0&2u
	\end{pmatrix}.\end{gathered}\]
	Consequently:
	\begin{align*}
		L[w]=\aij \left(u_{(\zeta) (\zeta) i \bar{j}} +4 \Re(\zeta_0) u_{(\zeta) i \bar{j}} +|2\Re(\zeta_0)|^2u_{i\bar{j}} \right)  \\
		+8r \Re(\zeta_0)\aij u_{i\bar{j}} + 4r \aij u_{(\zeta)i\bar{j}}+2|r|^2 \aij u_{i\bar{j}} + 2uq_i\aij \bar{q_j}+4r\Re(\aij \bar{q_j}u_i) \\
		+4\Re(\aij \overline{q}_j u_{(\zeta) i})+8 \Re(\aij \bar{q_j} u_i) -2\Re(\aij \bar{q_j}(2u_{(\zeta)i}+4 \Re(\zeta_0) u_i)).
	\end{align*}
	The terms on the last line cancel each other from the very definition of $b_j$ . The first two lines can be estimated, using concavity of the operator and bounding the last two terms of the second line in a trivial way, by
	\begin{equation}\label{Lw}\begin{gathered}
	L[w] \geq \frac{f_{(\zeta) ({\zeta})}+\left(|2\Re(\zeta_0)|^2+8\Re(r)\Re(\zeta_0)+2|r|^2\right)f + \left(4\Re(\zeta_0)+4\Re (r)\right) f_{(\zeta)}}{\sum_{l=1}^n\frac{\partial \mathcal F(D_{\cc}^2(u))}{\partial u_{l\bar{l}}}}\\
    -C|\zeta|^2 \geq -\mu||\zeta'||^2,\end{gathered}
	\end{equation}
	where we used the fact that $\sum_{l=1}^n\frac{\partial \mathcal F(D_{\cc}^2(u))}{\partial u_{l\bar{l}}}$ is bounded from below. This proves the claim.
\end{proof}

Using this time the barrier function 
\begin{equation} v = v_1+v_2+\frac\beta4 v_3 \end{equation} 
for $v_1$, $v_2$ and $v_3$ from the previous section we obtain:

\begin{theorem}\label{weakinteriorbound}
Let $\Om$ and $\psi$ be as in the beginning of the section. Let the $\Gamma$-admissible function $u$, continuous up to the boundary of $\Omega$, solve the equation
	\begin{equation}
	\begin{cases}\mathcal F(u_{z_i\bar{z}_j})=f,\\
		u|_{\partial \Om } = g.
	\end{cases} 	\end{equation}
Assume that $f,g\in C^{1,1}(\overline{\Om})$ and $f\geq 0$ in $\Omega$. Then for any $\epsilon>0$ and $\delta>0$ there is a constant $C_{\epsilon,\delta}=C\left(||f||_{C^{1,1}}, ||g||_{C^{1,1}}, ||\psi||_{C^{3,1}}, n, \epsilon, \delta \right)${, but independent of $\inf f$,} such that for any unit length vector $\zeta\in \mathbb C^n$ and $z\in \Omega_\delta=\{z \in \Omega \: | \: dist(z,\partial \Omega) > \delta\}$ one has
\begin{equation}
|u_{(\zeta)(\zeta)}| \leq \epsilon \sup_{\partial \Omega} |u_{\nu \nu}| + C_{\epsilon,\delta},
 \end{equation}
where $\nu$ is the inner normal along $\partial \Omega$.
\end{theorem}
\begin{proof}
Note that, again, due to Krylov's maximum principle it is enough to make the following two estimations. This time
\begin{equation}
L[v]=L[v_1+v_2+\frac\beta4 v_3]
\end{equation}
which coupled with Lemma \ref{barrier} and Lemma \ref{Lu} results in
\begin{equation}\label{firsttermterm}
\frac{-L[w]}{-L[v]} \leq  \frac{\mu \|\zeta'\|^2}{-L[v_1+v_2+\frac\beta4 v_3]} \leq C \cdot \big( \sup_{\overline{\Omega}}\psi+\beta\big)\beta^{-1}:= C_\beta.
\end{equation}
Next, let $z \in \partial \Omega$ and $|(\zeta,\zeta_0)|=1$ be such that: 
\begin{equation} \label{secondtermterm} 
\sup\limits_{\partial \Omega^N}\frac wv = \frac{w}{v}(z,\zeta'). \end{equation}
Let 
\begin{equation} \big|(\Re( \zeta), \nu)\big|=\theta.\end{equation}

We have, due to estimates of Section \ref{initial} and assumptions on $\psi$ on $\partial \Omega$, that at $(z,\zeta')$:
\begin{equation}
\begin{gathered}
w(z,\zeta') \leq 4\theta^2 \sup_{\partial \Omega} |u_{\nu \nu}| + C_\beta,\\
v(z,\zeta') \geq \theta^2 \beta^{-\alpha}+ C_\beta,
\end{gathered}
\end{equation}
where in $C_{\beta}$ we have included the tangential-tangential and tangential normal derivatives of $w$, which are under control.

This results in 
\begin{equation}\label{secondterm}
\sup\limits_{\partial \Omega^N}\frac wv \leq \epsilon \sup_{\partial \Omega} |u_{\nu \nu}| + C_\epsilon
\end{equation}
for $\epsilon>0$ as small as we wish at the cost of making $\beta$ sufficiently small and, potentially, $C_\epsilon$ big.
Applying (\ref{firsttermterm}) and (\ref{secondtermterm}) in (\ref{krylovmax}) for $\zeta_0=0$ and any norm one $\zeta$ we get for all $z \in \Omega$:
\begin{equation} \label{almost}
u_{(\zeta)(\zeta)}(z) \leq v(z,\zeta) \Big( \epsilon \sup_{\partial \Omega} |u_{\nu \nu}| + C_\epsilon \Big)\leq \frac{\epsilon}{\psi^{\alpha}}  \sup_{\partial \Omega} |u_{ \nu \nu}| + C_\epsilon. 
\end{equation}
For any $\delta>0$ and $\Omega_\delta = \{z \in \Omega \: | \: dist(z,\partial \Omega) > \delta\}$, as $\Omega_\delta \subset \subset \Omega$, we choose $\delta'$ such that 
\begin{equation}
\Omega_\delta \subset \{z \in \Omega \: | \: \psi^{\alpha}(z)>\delta'\}.
\end{equation}
Using (\ref{almost}) with $\epsilon':=\delta' \cdot \epsilon$ for any $\epsilon >0$ we obtain
\begin{equation}
u_{(\zeta)(\zeta)}(z) \leq  \epsilon  \sup_{\partial \Omega} |u_{ \nu \nu}| + C_{\epsilon, \delta}
\end{equation}
for any unit length $\zeta$ and $z \in \Omega_\delta$.
\end{proof}

\section{Second order estimate}
\subsection{Affine skew-hermitian vector fields}
The goal of this section is to prove a complex analogue of Lemma 2.1 from \cite{ITW}. First, we need a definition:
\begin{definition}\label{sekwherm}
 A holomorphic vector field $X$ is called affine skew-hermitian if
 \begin{equation}
 \xi=\sum_{j=1}^na_j\frac{\de}{\de z_j}+\sum_{k,l=1}^nA_{kl}z_{l}\frac{\de}{\de z_k}\end{equation}
 for constant $a_j \in \cc$ and the skew-hermitian matrix $A=(A_{kl})$, i.e. $A^*=-A$.
\end{definition}
Throughout this section we fix such a matrix $A$ and denote the corresponding holomorphic vector field by $\xi$.

The next lemma collects the commutation relations for differentiation with respect to such vector fields. It will be fundamental for the proof of {Proposition \ref{fquadratic}} below.
\begin{lemma}\label{commute}
 Let $u$ be {a} $C^4$ smooth function, possibly complex valued, and 
 \[ \xi=\xi_j\frac{\de}{\de z_j}\] an affine skew-hermitian vector field. Then
 \begin{equation} \label{coefder}
  (\xi_k)_l=A_{kl}=-\bar{A}_{lk}=-(\bar{\xi}_l)_{\bar{k}}, \end{equation}
  \begin{equation} \label{third} u_{(\xi)i\bar{j}}=u_{i\bar{j}(\xi)}+(\xi_k)_iu_{k\bar{j}}+(\bar{\xi}_k)_{\bar{j}}u_{i\bar{k}}, \end{equation}
  \begin{equation} \label{fourth} \begin{gathered}
  u_{(\xi)(\xi)i\bar{j}}=
  u_{i\bar{j}(\xi)(\xi)}\\
  +(\xi_l)_i(\xi_k)_lu_{k\bar{j}}+(\bar{\xi}_l)_{\bar{j}}(\bar{\xi}_k)_{\bar{l}}u_{i\bar{k}}+(\bar{\xi_l})_{\bar{j}}(\xi_k)_iu_{k\bar{l}}+(\xi_l)_i(\bar{\xi}_k)_{\bar{j}}u_{l\bar{k}}\\
  +(\xi_k)_iu_{k\bar{j}(\xi)}+(\xi_l)_iu_{l\bar{j}(\xi)}+(\bar{\xi}_l)_{\bar{j}}u_{i\bar{l}(\xi)}+(\bar{\xi}_k)_{\bar{j}}u_{i\bar{k}(\xi)}.
 \end{gathered} \end{equation}
\end{lemma}
\begin{proof}
Formula (\ref{coefder}) just exploits the affine skew-hermitian property of $\xi$. Formula (\ref{third}) follows from the holomorphicity of $\xi$ and no further properties of the coefficients are involved. For (\ref{fourth}) one makes the iterative application of (\ref{third}) to get
\begin{equation}
\begin{gathered}
u_{(\xi)(\xi)i\bar{j}}=
u_{(\xi)i\bar{j}(\xi)}+(\xi_k)_iu_{(\xi)k\bar{j}}+(\overline{\xi_k})_{\bar{j}} u_{(\xi)i\bar{k}}\\
=u_{i\bar{j}(\xi)(\xi)}+(\xi_k)_iu_{k\bar{j}(\xi)}+(\overline{\xi_k})_{\bar{j}} u_{i\bar{k}(\xi)}\\
+(\xi_k)_i\left( u_{k\bar{j}(\xi)}+(\xi_l)_ku_{l\bar{j}}+(\overline{\xi_l})_{\bar{j}} u_{k\bar{l}(\xi)}\right)\\
+(\xi_{\bar{k}})_{\bar{j}}\left(u_{i\bar{k}(\xi)}+(\xi_l)_iu_{l\bar{k}}+(\overline{\xi_l})_{\bar{k}} u_{i\bar{l}} \right)
\end{gathered}
\end{equation}
Note that we heavily used the complex affine structure of $\xi$ in the computation above.
\end{proof}

The advertised complex analogue of Lemma 2.1 from \cite{ITW} turns out to be true.
\begin{proposition}\label{fquadratic}
For any complex Hessian operator $\mathcal F$, affine skew-hermitian vector field $\xi$ and admissible function $u$ we have:
\begin{equation}\label{seven'}
 \mathcal F^{i\bar{j}}u_{(\xi)(\xi)i\bar{j}}=\Big(\mathcal F\big((u_{p\bar{q}})\big)\Big)_{(\xi)(\xi)}-\mathcal F^{i\bar{j},p\bar{q}}u_{(\xi)i\bar{j}}u_{(\xi)p\bar{q}},
\end{equation}
in particular for concave $\mathcal F$ we get
\begin{equation}\label{seven}
 \mathcal F^{i\bar{j}}u_{(\xi)(\xi)i\bar{j}}\geq \Big(\mathcal F \big((u_{p\bar{q}})\big) \Big)_{(\xi)(\xi)}.
\end{equation}
\end{proposition}
\begin{proof} 
Recall the basic fact from linear algebra that for any skew-hermitian matrix $A$ and any $t\in\mathbb R$ the matrix $e^{t\bar{A}}$ is {\it unitary}. Of course, the same is true for $e^{tA}$. 

Thus given any complex Hessian operator $\mathcal F$ applied to the complex hermitian matrix $B$ we have for any $t\in\mathbb R$:
\begin{equation} \label{invar}
 \mathcal F(B)=\mathcal F(e^{t\bar{A}}Be^{-t\bar{A}}).
\end{equation}

Differentiating (\ref{invar}) for
\begin{equation}
A=(A_{ab}) \text{ and } B=(u_{i\bar{j}})
\end{equation}
once and twice at $t=0$ and using (\ref{coefder}) results in the following formulas, cf. (2.6) in \cite{ITW}:
\begin{equation}\label{one}
 \mathcal F^{i\bar{j}}(\overline{A_{ik}}u_{k\bar{j}} -u_{i\bar{k}}\overline{A_{kj}})=
 \mathcal F^{i\bar{j}}\big((\bar{\xi}_i)_{\bar{k}}u_{k\bar{j}}+(\xi_j)_ku_{i\bar{k}}\big)=0,
\end{equation}
\begin{equation}\label{two} \begin{gathered}
\mathcal F^{i\bar{j}}\big(A_{li}A_{kl}u_{k\bar{j}} +u_{i\bar{k}}A_{lk}A_{jl} -2A_{ki}u_{k\bar{l}}A_{jl}\big)\\
 +\mathcal F^{i\bar{j},p\bar{q}}
\big((\bar{\xi}_i)_{\bar{k}}u_{k\bar{j}}+(\xi_j)_ku_{i\bar{k}}\big)\big((\bar{\xi}_p)_{\bar{l}}u_{l\bar{q}}+(\xi_q)_lu_{p\bar{l}}\big)=0. \end{gathered}
\end{equation}

Differentiating, by $(\xi)$, both sides of (\ref{one}) results in
\begin{equation}\label{three}
\mathcal F^{i\bar{j}}\big((\bar{\xi}_i)_{\bar{k}}u_{k\bar{j}(\xi)}+(\xi_j)_ku_{i\bar{k}(\xi)}\big)+\mathcal F^{i\bar{j},p\bar{q}}\big((\bar{\xi}_i)_{\bar{k}}u_{k\bar{j}}+(\xi_j)_ku_{i\bar{k}}\big)u_{p\bar{q}(\xi)}=0
\end{equation}
where we used the fact that the derivatives of the coefficients of $\xi$ are constants.

To proceed further we exploit the commutation formulas (\ref{fourth}) for switching partial derivatives and differentiation with respect to $(\xi)$ in order to obtain:
\begin{equation}\label{four} \begin{gathered}
 \mathcal F^{i\bar{j}}u_{(\xi)(\xi)i\bar{j}}=\mathcal F^{i\bar{j}}u_{i\bar{j}(\xi)(\xi)}
  \\+\mathcal F^{i\bar{j}} \Big((\bar{\xi}_l)_{\bar{j}}(\bar{\xi}_k)_{\bar{l}}u_{i\bar{k}} +(\xi_l)_i(\xi_k)_lu_{k\bar{j}} +(\xi_k)_i(\bar{\xi}_l)_{\bar{j}}u_{k\bar{l}} +(\bar{\xi}_k)_{\bar{j}}(\xi_l)_iu_{l\bar{k}}\\
 +(\xi_k)_iu_{k\bar{j}(\xi)}+(\xi_l)_iu_{l\bar{j}(\xi)}+(\bar{\xi}_l)_{\bar{j}}u_{i\bar{l}(\xi)}+(\bar{\xi}_k)_{\bar{j}}u_{i\bar{k}(\xi)}\Big).\end{gathered}
 \end{equation}
We denote:
\begin{equation}\begin{gathered}
 C=\mathcal F^{i\bar{j}}u_{i\bar{j}(\xi)(\xi)}, \\
 D=\mathcal F^{i\bar{j}}\Big((\bar{\xi}_l)_{\bar{j}}(\bar{\xi}_k)_{\bar{l}}u_{i\bar{k}} +(\xi_l)_i(\xi_k)_lu_{k\bar{j}} +(\xi_k)_i(\bar{\xi}_l)_{\bar{j}}u_{k\bar{l}} +(\bar{\xi}_k)_{\bar{j}}(\xi_l)_iu_{l\bar{k}}\\
  +(\xi_k)_iu_{k\bar{j}(\xi)}+(\xi_l)_iu_{l\bar{j}(\xi)}+(\bar{\xi}_l)_{\bar{j}}u_{i\bar{l}(\xi)}+(\bar{\xi}_k)_{\bar{j}}u_{i\bar{k}(\xi)}\Big).\end{gathered}
 \end{equation}
The next observation is that using (\ref{commute}) and (\ref{third}):
\begin{equation} \begin{gathered}
 C=\Big(\mathcal F\big((u_{p\bar{q}})\big)\Big)_{(\xi)(\xi)}-\mathcal F^{i\bar{j},p\bar{q}}u_{i\bar{j}(\xi)}u_{p\bar{q}(\xi)}\\
 = \Big(\mathcal F\big((u_{p\bar{q}})\big)\Big)_{(\xi)(\xi)}
 -\mathcal F^{i\bar{j},p\bar{q}}(u_{(\xi)i\bar{j}}+(\bar{\xi}_i)_{\bar{k}}u_{k\bar{j}}+({\xi}_j)_{{k}}u_{i\bar{k}})u_{p\bar{q}(\xi)}.\end{gathered}
\end{equation}
Computing further
\begin{equation}\label{five} \begin{gathered}
 C=\Big(\mathcal F\big((u_{p\bar{q}})\big)\Big)_{(\xi)(\xi)}-\mathcal F^{i\bar{j},p\bar{q}}u_{(\xi)i\bar{j}}u_{p\bar{q}(\xi)}+\mathcal F^{i\bar{j}}((\bar{\xi}_i)_{\bar{k}}u_{k\bar{j}(\xi)}+(\xi_j)_ku_{i\bar{k}(\xi)})\\
=\Big(\mathcal F\big((u_{p\bar{q}})\big)\Big)_{(\xi)(\xi)}-\mathcal F^{i\bar{j},p\bar{q}}u_{(\xi)i\bar{j}}u_{(\xi)p\bar{q}}-\mathcal F^{i\bar{j},p\bar{q}}u_{(\xi)i\bar{j}}((\bar{\xi}_p)_{\bar{l}}u_{l\bar{q}(\xi)}+(\xi_q)_lu_{p\bar{l}(\xi)})\\
+\mathcal F^{i\bar{j}}((\bar{\xi}_i)_{\bar{k}}u_{k\bar{j}(\xi)}+(\xi_j)_ku_{i\bar{k}(\xi)})\\
=\Big(\mathcal F\big((u_{p\bar{q}})\big)\Big)_{(\xi)(\xi)}-\mathcal F^{i\bar{j},p\bar{q}}u_{(\xi)i\bar{j}}u_{(\xi)p\bar{q}}
 +2\mathcal F^{i\bar{j}}((\bar{\xi}_i)_{\bar{l}}u_{l\bar{j}(\xi)}+(\xi_j)_lu_{i\bar{l}(\xi)})\\
+\mathcal F^{i\bar{j},p\bar{q}}((\bar{\xi}_i)_{\bar{k}}u_{k\bar{j}}+({\xi}_j)_{{k}}u_{i\bar{k}})((\bar{\xi}_p)_{\bar{l}}u_{l\bar{q}(\xi)}+(\xi_q)_lu_{p\bar{l}(\xi)})\\
=\Big(\mathcal F\big((u_{p\bar{q}})\big)\Big)_{(\xi)(\xi)}-\mathcal F^{i\bar{j},p\bar{q}}u_{(\xi)i\bar{j}}u_{(\xi)p\bar{q}}
 +2\mathcal F^{i\bar{j}}((\bar{\xi}_i)_{\bar{l}}u_{l\bar{j}(\xi)}+(\xi_j)_lu_{i\bar{l}(\xi)})\\
-\mathcal F^{i\bar{j}}\big((\xi_l)_i(\xi_k)_lu_{k\bar{j}}+(\bar{\xi}_k)_{\bar{l}}(\bar{\xi}_l)_{\bar{j}}u_{i\bar{k}}+2(\xi_k)_i(\bar{\xi}_l)_{\bar{j}}u_{k\bar{l}}\big)\end{gathered}
\end{equation}
where in the calculations we used twice (\ref{three}) and then (\ref{two}) to get the above equalities.

On the other hand using (\ref{coefder}) and the fact that $\xi$ is skew-hermitian the term $D$ reads
\begin{equation}\label{six}\begin{gathered}
 D=
 \mathcal F^{i\bar{j}}(\bar{\xi}_{{l}})_{\bar{j}}(\bar{\xi}_{{k}})_{\bar{l}}u_{i\bar{k}}
 +\mathcal F^{i\bar{j}}({\xi}_l)_{{i}}({\xi}_k)_{{l}}u_{k\bar{j}}
 +2\mathcal F^{i\bar{j}}(\bar{\xi}_l)_{\bar{j}}(\xi_k)_{i}u_{k\bar{l}}\\
- 2 \mathcal F^{i\bar{j}}\big((\bar{\xi}_i)_{\bar{l}}u_{l\bar{j}(\xi)}+(\xi_j)_lu_{i\bar{l}(\xi)}\big).\end{gathered}\end{equation}

Coupling (\ref{five}) and (\ref{six}) we finally obtain the formula (\ref{seven'}). As $u_{(\xi)}$ is real valued function and $\mathcal F$ is assumed to be concave we get (\ref{seven}) as well.
\end{proof}

\subsection{Skew-hermitian approximate tangential vector fields}

For the proof of existence of special vector fields described in the Introduction we need special holomorphic coordinates. 
\begin{definition}
 Let $\Omega$ be a bounded strictly pseudoconvex domain in $\C$ with a boundary of class $\mathcal C^4$ and fix a point $p\in\partial\Omega$. 
 
A local coordinate system centered at $p$, which we always identify with the coordinate origin, is said to be adapted if in these coordinates, around $p$, the defining function $-\psi$ reads
\begin{equation} \label{rho} -\psi(z)=\rho(z)=-2\Re(z_n)+\sum_{j=1}^n|z_j|^2+R_3(z)+O(||z||^4),\end{equation}
\end{definition}
where $R_3$ is a homogeneous third order term. The only role $R_3$ plays later on is that its derivatives are $O(||z||^2)$. Henceforth, for the sake of simplicity, the expansion will be written with $R_3(z)+O(||z||^4)$ replaced by $O(||z||^3)$, tacitly understanding that its first order derivatives are $O(||z||^2)$ terms.

Note that in particular in adopted coordinates around origin
\begin{equation}
\label{xnasymp} x_n = O(\xn)
\end{equation}
for $z'=(z_1,...,z_{n-1})$.

Note that for the unit ball at the point $(0,\cdots,0,1)$ the defining function 
\[\rho(z):=||z||^2-1\] has exactly the shape above without the third order remainder. We shall need the following classical fact which can be traced back to \cite{BT76}.
\begin{proposition}\label{bedtay}
 Adapted coordinates exist near any boundary point $p$ of a $C^4$ strictly pseudoconvex domain.
\end{proposition}
For $(1,0)$ vector fields $\xi=\xi_l\frac{\partial}{\partial z_l}$ and $\eta=\eta_l\frac{\partial}{\partial z_l}$ we introduce the notation standing to the end of the paper:
\begin{equation}
\begin{gathered}
(\xi,\eta):= \Re \left(\sum_{l=1}^n\xi_l\overline{\eta_l}\right)=4\left(\Re(\xi),\Re(\eta)\right),\\
\left\langle \xi,\eta \right\rangle = \sum_{l=1}^n\xi_l\overline{\eta_l},\\
||\xi||=\left\langle \xi,\xi\right\rangle^{\frac{1}{2}}.
\end{gathered}
\end{equation}
Note that $(\cdot,\cdot)$ denotes the standard inner product in $\rr^{2n}$.
Given a domain $\Omega$ in $\C$  described by a defining function $\rho$ we introduce the notion of an approximate tangential vector field.
\begin{definition}\label{tang}
A $(1,0)$ vector field $\xi$ is said to be approximate (unital) tangential near a point $p\in\partial\Omega$ in a coordinate system $(z_1,\cdots,z_n)$ centered at $p$ if one has
 \begin{enumerate}
  \item $(\xi,\eta)=\Re(\xi_j(z)\frac{\de\rho}{\de z_j}(z))=O(\xn)$ for $\eta=\frac{\de\rho}{\de \overline{z_i}}\partial_{z_i}$;
  \item If $\zeta(z)$ is the $(1,0)$ part of the projection of $\Re(\xi(z))$ onto the (real) tangent space $T_{z}^{\mathbb R}\partial\Omega$ then $||\xi(z)-\zeta(z)||\leq C \xn$ for some $C$ dependent only on $\partial\Omega$;
  \item ${||\xi(z)||=1+O(\xn)}$.
 \end{enumerate}
\end{definition}

In this section we prove the following result:
\begin{proposition}\label{approx}
 Let $\Omega$ be a bounded strictly pseudoconvex domain in $\C$ with a boundary of class $\mathcal C^4$. Then near any $p\in\partial\Omega$ in an adapted coordinate system there exist $2n-1$ holomorphic affine skew-hermitian vector fields which are approximate tangential and span $T_z^{\mathbb R}\partial\Omega$ at $p$. Namely, we can choose
\begin{equation}\label{xk}
 \xi_k=\de_{z_k}+z_k\de_{z_n}-z_n\de_{z_k},\ \ k=1,\cdots,n-1
\end{equation}
\begin{equation}\label{yj}
 \xi_{n-1+j}=i\de_{z_j}-iz_j\de_{z_n}-iz_n\de_{z_j},\ \ j=1,\cdots,n-1
\end{equation}
\begin{equation}\label{yn}
 \xi_{2n-1}=i\de_{z_n}-iz_n\de_{z_n}.\ \
\end{equation}
\end{proposition}
\begin{proof}
All these vector fields satisfy
\[ ||\xi_s||=1+O(||(z',y_n)||^2)\] which can be easily seen as $\Re(z_n)$ itself is an $O(||(z',y_n)||^2)$ term. Thus $(3)$ is confirmed. Regarding $(1)$ we have for any $1 \leq s \leq n-1$
\begin{equation}\begin{gathered}
\frac{\de\rho}{\de z_j}(z) \cdot (\xi_s)_j(z)=\rho_s(z)(1-z_n)+\rho_n(z) z_s \\
=(\rho_{s\bar{s}}(0)\overline{z_s}+O(||z||^2))(1-z_n)+(\rho_n(0)+\rho_{n\bar{n}}(0)\overline{z_n})z_s\\
=-2\ii\Im(z_s)+O(||z||^2).
\end{gathered}\end{equation}
Next, for $1\leq s \leq n-1$
\begin{equation}\begin{gathered}
\frac{\de\rho}{\de z_j}(z) \cdot (\xi_{s+n-1})_j(z)=
\ii \rho_s(z)(1-z_n)-\ii\rho_n(z) z_s \\
=\ii(\rho_{s\bar{s}}(0)\overline{z_s}+O(||z||^2))(1-z_n)-\ii(\rho_n(0)+\rho_{n\bar{n}}(0)\overline{z_n})z_s\\
=2\ii\Re(z_s)+O(||z||^2)
\end{gathered}\end{equation}
so we have verified $(1)$ for (\ref{xk}) and (\ref{yj}). At last
\begin{equation}\begin{gathered}
\frac{\de\rho}{\de z_j}(z) \cdot (\xi_{2n-1})_j(z)=\rho_n(z)\ii(1-z_n)\\
=\ii(\rho_n(0)+\rho_{n\bar{n}}(0)\overline{z_n}+O(||z||^2))(1-z_n)\\
=-\ii|1-z_n|^2+O(||z||^2).
\end{gathered}\end{equation}
To verify $(2)$ we note that the real part of
\begin{equation}\label{normalvf}
\eta = \rho_{\overline{z_i}}\partial_{z_i}
\end{equation}
is the local normal vector field along $\partial \Omega$. We decompose
\begin{equation}
\begin{gathered}
\xi = ( \xi,\eta ) \cdot \eta + \zeta
\end{gathered}
\end{equation}
meaning in particular that $\Re(\zeta)$ is the projection of $\Re(\xi)$ on tangent vector space. Thus the difference to be examined equals
\begin{equation}
||\xi(z)-\zeta(z)|| = ||( \xi,\eta ) \cdot \eta||=|( \xi,\eta )|\cdot|| \eta||=O(||z||^2)
\end{equation} 
from property $(1)$ which we just verified above. We note that $(2)$ is, in fact, a consequence of $(1)$ and is included in the definition for the special role it plays in the argument below.
\end{proof}

\subsection{Boundary estimate}
We recall that for any affine-hermitian approximate tangential vector field $\xi$ we apply the decomposition
\begin{equation}
\xi = ( \xi,\eta ) \cdot \eta + \zeta
\end{equation}
where $\Re(\zeta)$ is tangent vector field and $\eta$ for the rest of the section is as in (\ref{normalvf}). For convenience we also introduce the unit normal vector field
\[\label{realnormal} \gamma = - \frac{\Re(\eta)}{\left(\Re(\eta),\Re(\eta)\right)^{
\frac{1}{2}}}.\] Of course in adopted coordinates at $p$ \[\gamma(p)=\partial_{x_n}.\]

Starting from this subsection, we make use of the following additional notation:
\begin{equation}\label{M}
{M:=\sup_{\partial\Omega}u_{\gamma\gamma}.}
\end{equation}
We shall also use the following notation, borrowed from \cite{ITW}. A function $h$ will denote any function of the type
\begin{equation}\label{hex} h(z',y_n)={a}\begin{pmatrix}
		x_j\\
		y_j\\
		y_n
	\end{pmatrix}+C_1||(z',y_n)||^2+C_2M||(z',y_n)||^4,\end{equation}
where $j\in \{ 1,\cdots,n-1\}$ and {$a$} is a $(2n-1)$ - dimensional vector with entries uniformly under control if $z$ is sufficiently close to $0$. Moreover, $C_1$ and $C_2$ are also uniform constants under control. In particular, the shape of $h$ may vary line-to-line but the pertinent constants will always be under control, they will depend on the geometry of $\partial\Omega$ and on the bounds of $u$ established in Subsection 3.2.

Let us start with the analysis of quantity
\begin{equation}
w(z)=u_{(\xi)(\xi)}(z)-u_{(\xi)(\xi)}(0)
\end{equation}
for any $\xi$ from Proposition \ref{approx}. One computes, at a  point $z$,
\begin{equation} \label{(xi)(xi)}
\begin{gathered}
u_{(\xi)(\xi)}= u_{[\xi][\xi]}+ \xi_i (\xi_j)_i u_j+\xi_i(\xi_{\bar{j}})_iu_{\bar{j}}+\xi_{\bar{i}}(\xi_j)_{\bar{i}}u_j+\xi_{\bar{i}}(\xi_{\bar{j}})_{\bar{i}}u_{\bar{j}}\\
=u_{[\xi][\xi]}+ \xi_i (\xi_j)_i u_j+\xi_{\bar{i}}(\xi_{\bar{j}})_{\bar{i}}u_{\bar{j}}
= u_{[\xi][\xi]}+ 2 \Re \big(\xi_i (\xi_j)_i u_j\big).
\end{gathered}
\end{equation}

We now compute
\begin{equation} \label{[xi][xi]}
\begin{gathered}
u_{[\xi][\xi]}=u_{[( \xi,\eta ) \cdot \eta + \zeta][( \xi,\eta ) \cdot \eta + \zeta]}\\
=( \xi,\eta )^2 u_{[\eta ][\eta]}+2( \xi,\eta ) u_{[\eta ][\zeta]}+u_{[\zeta][\zeta]}:=B_1+B_2+B_3.
\end{gathered}
\end{equation}

One easily sees:
\begin{equation}\label{bbest}
\begin{gathered}
B_1 \leq CM ||(z',y_n)||^4, \\
B_2 \leq C ||(z',y_n)||^2
\end{gathered}
\end{equation}
as from Proposition \ref{approx} we have $( \xi,\eta ) = O(||(z',y_n)||^2)$.

Using the boundary condition for $u$ and the fact that $\Re (\zeta)$ is a tangent vector field, we obtain the following: 
\begin{equation} \label{curvsec} \begin{gathered}
B_3=u_{[\zeta][\zeta]}=4 u_{\Re (\zeta) \Re (\zeta)}= 4 \phi_{\Re (\zeta) \Re (\zeta)}\\
+4\left(\Re (\zeta),\Re (\zeta)\right) \kappa\left(z,\frac{\Re (\zeta)}{\left(\Re (\zeta),\Re (\zeta)\right)^{\frac{1}{2}}}\right)(\phi-u)_{\gamma}\\
=\phi_{[\zeta][\zeta]}+||\zeta||^2 \kappa\left(z,\frac{\Re (\zeta)}{\left(\Re (\zeta),\Re (\zeta)\right)^{\frac{1}{2}}}\right)(\phi-u)_{\gamma}\\
\leq \phi_{[\zeta][\zeta]}+ \kappa\left(z,\frac{\Re (\zeta)}{\left(\Re (\zeta),\Re (\zeta)\right)^{\frac{1}{2}}}\right)(\phi-u)_{\gamma} + h(z)\end{gathered},
\end{equation} 
where $\kappa\left(z,\frac{\Re (\zeta)}{\left(\Re (\zeta),\Re (\zeta)\right)^{\frac{1}{2}}}\right)$ denotes the curvature in direction $\frac{\Re (\zeta)}{\left(\Re (\zeta),\Re (\zeta)\right)^{\frac{1}{2}}}$. Note that we used $||\zeta||^2 = 1 + O(||(z',y_n)||^2)$ in the above calculation.

Returning to the term $2 \Re \big(\xi_i (\xi_j)_i u_j\big)$ in (\ref{(xi)(xi)}) we get for $\xi=\xi_k$ from Proposition \ref{approx} 
\begin{equation} \label{xixifirst}
\begin{gathered}
\xi_i (\xi_j)_i u_j = \begin{cases}
\big( \delta_i^k(1-z_n)+\delta_i^nz_k\big) \big(\ \delta_j^k(1-z_n)+\delta_j^n z_k\big)_i u_j
\\ =-z_ku_k+(1-z_n)u_n \text{ for } k \in \{1,...,n-1\};\\
\big( \delta_i^{k-n+1}\ii(1-z_n)+\delta_i^n(-\ii z_{k-n+1})\big) \big(\ \delta_j^{k-n+1}\ii(1-z_n)\\
+\delta_j^n (-\ii z_{k-n+1})\big)_i u_j =-z_{k-n+1}u_{k-n+1}+(1-z_n)u_n \\
\text{ for } k \in \{n,...,2n-2\};\\
\big(\delta^n_i\ii (1-z_n) \big) \big(\delta^n_j \ii (1-z_n) \big)_i u_j \\
=(1-z_n)u_n \text{ for } k =2n-1.
\end{cases}
\end{gathered}
\end{equation}

We can now finish the preliminary analysis of the quantity $w$ with the following lemma.
\begin{lemma}\label{w-bound}
We have for $z\in\partial\Omega$, sufficiently close to $0$
$$w(z)\leq h(z',y_n).$$
\end{lemma}
\begin{proof}
Recall that $h$ will change from line to line here, but all such $h$'s will satisfy (\ref{hex}). From (\ref{(xi)(xi)}) and (\ref{[xi][xi]}) after applying (\ref{bbest}) we obtain
\begin{equation} \label{wdiff1}
w(z) \leq u_{[\zeta][\zeta]}(z) - u_{[\zeta][\zeta]}(0)+ 2 \Re \left(\xi_i (\xi_j)_i u_j \right)- 2 \Re \left(\xi_i (\xi_j)_i u_j \right)(0)+h(z',y_n).
\end{equation}
Applying (\ref{curvsec}) we may estimate further in (\ref{wdiff1})
\begin{equation} \label{wdiff2} \begin{gathered}
w(z) \leq \phi_{[\zeta][\zeta]}(z) - \phi_{[\zeta][\zeta]}(0)+K \left(z,\widetilde{\Re {\zeta}}\right)(\phi-u)_{\gamma}-K \left(0,\widetilde{\Re {\zeta}(0)}\right)(\phi-u)_{\gamma}(0)
\\+ 2 \Re \left(-z_ku_k+(1-z_n)u_n \right)- u_{x_n}(0)+h(z',y_n) \\\leq 
K \left(0,\widetilde{\Re {\zeta}(0)}\right)u_{x_n}(z)-K \left(z,\widetilde{\Re {\zeta}}\right)u_{\gamma}(z)\\+ u_\gamma(0)-u_{x_n}(0)+h(z',y_n),
\end{gathered} \end{equation}
{where the wave denotes unit length normalization of the vector.}
One observes that:
\begin{equation} \label{wdiff3} \begin{gathered}
u_{x_n}(z)-u_{\gamma}(z)\leq a' \cdot (z',y_n) + C \xn \leq h(z',y_n)
\end{gathered} \end{equation} for $a'$ and $C$ under control.
This is the case as, due to (\ref{rho}),
\begin{equation} \begin{gathered}
\rho_{{x_k}} = a' \cdot (z',y_n) + O( \xn) \text{ for } k<n, \\
\rho_{{y_k}} = a' \cdot (z',y_n) + O( \xn)  \text{ for } k<n.
\end{gathered}
\end{equation}
Thus in order to confirm (\ref{wdiff3}) it is enough to check that
\begin{equation}
u_{x_n}(z)-\gamma^{x_n}(z)u_{x_n}(z) \leq C \xn,
\end{equation}
where $\gamma^{x_n}$ denotes the coefficient in the frame $\partial_{x_i},\partial_{y_i}$.
This in turn follows again using (\ref{rho}) from
\begin{equation}
1-\gamma^{x_n}(z )\leq O(\xn).
\end{equation}
Thus, applying (\ref{wdiff3}) in (\ref{wdiff2}) results in 
\begin{equation} \label{wdiff2} \begin{gathered}
w(z) \leq 
\kappa \left(0,\widetilde{\Re {\zeta}(0)}\right)u_{x_n}(z)-\kappa \left(z,\widetilde{\Re {\zeta}}\right)u_{x_n}(z)+h(z',y_n)\\ \leq 
\left(\kappa \left(0,\widetilde{\Re {\zeta}(0)}\right)-\kappa \left(z,\widetilde{\Re {\zeta}}\right)\right)\left(u_{x_n}(z)-u_{x_n}(0)\right)+h(z',y_n) \leq h(z',y_n),
\end{gathered} \end{equation}
as required.
\end{proof}

In particular, by subtracting from $w$ the linear term in $h$ (this will not change the latter arguments), we can assume that for sufficiently small $r_0$ depending on the geometry of $\partial\Omega$ and on the gradient, tangential-tangential and tangential-normal bounds on $u$ (quantities which we already control) we have
\begin{equation}\label{w-boundfinal}
 w(z)\leq C_3 \big(||(z',y_n)||^2+M||(z',y_n)||^4\big),
\end{equation}
for all $z\in \partial \Omega$ such that $||(z',y_n)||\leq r_0$.

\subsection{Application of the maximum principle}

Assume that $w$ and $\xi$ are as above. In the following, we prove the complex analogue of Lemma 2.2 from \cite{ITW}:
\begin{lemma}\label{2.2itw}
There is a constant $C_4$ dependent on $C_3$ and $r_0$ such that for any $z\in\Omega$ near the origin the following inequality holds
\begin{equation}
w(z)\leq C_4\big(||(z',y_n)||^2+M||(z',y_n)||^4+M\psi(z)\big).
\end{equation} 
\end{lemma}
\begin{proof}
 The argument copies the reasoning from \cite{ITW}. We provide the details for the sake of completeness. Define for $r\in(0,r_0]$ the domain \[ \omega_r:=\lbrace z\in\Omega \: | \: \psi(z)<r^4,\ ||(z',y_n)||<r \rbrace .\]

Consider the barrier function
\[ v(z)=K[-\psi(z)-\frac{1}M||(z',y_n)||^2-||(z',y_n)||^4]\]
for $K>0$ to be determined.
Recall that due to definition \ref{pseudoconv} $-\psi$ is smooth and strictly plurisubharmonic up to the boundary. Hence, half of the complex Hessian of $-\psi$ dominates the Hessian of $-\frac{1}M||(z',y_n)||^2-||(z',y_n)||^4$ everywhere as long as $M$ is sufficiently large (making the complex Hessian of the first part sufficiently small) and $r$ is sufficiently small (taking care of the smallness of the Hessian of the second part). Thus, exploiting (\ref{apsi}) {and recalling that $\sum_{l=1}^n F^{ll}$ is bounded from below}, we can have for any $K>1$:
 \begin{equation}\label{Lbound}
  \mathcal L(v)\geq 1.
 \end{equation}
 
In particular, if $M$ is large enough (depending on $||f^{1/n}||_{C^2}$) we obtain using (\ref{seven}) in Proposition \ref{fquadratic}:
 \begin{equation}\label{Lbound}
  \mathcal L(Mv+w)\geq0.
 \end{equation}
 
 In order to apply the maximum principle, we need to check that \[Mv+w\leq 0\] on $\partial\omega_r$. 
But \[ \partial\omega_r=\partial_1\omega_r\cup\partial_2\omega_r\cup\partial_3\omega_3,\] 
where $\partial_1\omega_r=\partial\Omega\cap \partial \omega_r$, $\partial_2\omega_r=\lbrace z\in\ \Omega|\ \psi(z)=r^4\rbrace$ and finally \newline $\partial_3\omega_r=\lbrace z\in\ \Omega|\ ||(z',y_n)||=r\rbrace$.

Note that on $\partial_1\omega_r$
\begin{align*}
 &w(z)+Mv(z)\leq C_3(||(z',y_n)||^2+M||(z',y_n)||^4)\\
 &
 +MK(-\psi(z)-\frac{1}M||(z',y_n)||^2-||(z',y_n)||^4),
\end{align*}
where we used (\ref{w-boundfinal}). As $-\psi\leq 0$ it is obvious that \[ w+Mv\leq 0\] on $\partial_1\omega_r$ as long as $K \geq \max \{1,C_3\}$.

Moving on to $\partial_2\omega_r$ note that in this set:
\[ Mv(z)\leq -KMr^4-K||(z',y_n)||^2\] which beats $h(z',y_n)$ as long as $K\geq C_3$. Finally on $\partial_3\omega_r$ once again
$$Mv(z)\leq -KMr^4-Kr^2$$
and hence also on this piece of the boundary the inequality $Mv+w\leq 0$ holds under the previous assumption on $K$.

In conclusion \[ w\leq-Mv(z)\] on the whole $\omega_r$ which is exactly the claimed bound.
\end{proof}

We now turn to proving the bound on the normal derivative of $w$. This is an analogue of Lemma 2.3 from \cite{ITW}.

\begin{lemma} \label{tantann}
For any sufficiently small $r>0$ there exists a constant $C$ depending on $\psi$ and $C_r>0$ depending only on $C_{r^5,r^4}$ from Theorem \ref{weakinteriorbound} such that
\begin{equation}\label{ttn}
(u_{(\xi)(\xi)})_{x_n}(0) \leq r M + C_r.
\end{equation}
\end{lemma}
\begin{proof}
This time we set
\[v(z)=K \big( -\psi(z) - \beta ||(z',y_n)||^2\big)\]
for some $K,\beta>0$. We aim at verifying 
\begin{equation}\label{linest}
\mathcal{L}\big((rM+C_r)v+w\big) \geq 0,\end{equation}
\begin{equation}\label{boundest}
(rM+C_r)v+w \leq 0 \text{ on }\partial \omega_r 
\end{equation} for $\omega_r$ as before. 

Arguing as in Lemma \ref{2.2itw}, using in particular (\ref{apsi}) and (\ref{seven}) after choosing $\beta$ small enough in comparison to strict plurisubharmonicity of $-\psi$, $K>1$ and $C_r$ big enough in comparison to $||f^{1/n}||_{C^2}$ and $1$ we can arrange (\ref{linest}).

For the rest of the argument we take $C_r$ as the maximum of the value above and $r^{-4}C_{r^5,r^4}$ and continue to use the notation $\partial_i \omega_r$.

On $\partial_1 \omega_r$ utilizing Lemma \ref{2.2itw} we have 
\begin{equation}\label{om1}\begin{gathered} (rM+C_r)v+w \leq -K\beta rM \xn - K \beta C_r \xn \\ + C_4\big(||(z',y_n)||^2+M||(z',y_n)||^4\big).
\end{gathered}\end{equation}
After choosing $K$ big enough (compared to $\beta$ and $C_4$), we see that the first term on RHS of (\ref{om1}) beats the third one, while for $r$ uniformly small enough the second one dominates the fourth one giving (\ref{boundest}) for points on $\partial_1 \omega_r$.

On $\partial_2 \omega_r$ applying Theorem \ref{weakinteriorbound} to $u_{(\xi)(\xi)}(z)$ and the bound on $u_{(\xi)(\xi)}(0)$ we have 
\begin{equation}\label{om2}\begin{gathered} (rM+C_r)v+w \leq -K(rM+C_r)r^4+r^5M+C_{r^5,r^4}\\
\leq -r^5MK-KC_{r^5,r^4}+r^5M+C_{r^5,r^4} \leq 0,\\
\end{gathered}\end{equation}
where we used $K>1$.

On $\partial_3 \omega_r$  we have 
\begin{equation}\label{om3}\begin{gathered} (rM+C_r)v+w \leq -(rM+C_r)\beta K r^2 + 2 C_4\big(r^2+Mr^4\big),
\end{gathered}\end{equation}
which easily is seen to be non positive for $r$ uniformly small enough, and $C_r$ big enough (compared to $C_4$). This finishes the proof of (\ref{boundest}).

The maximum principle guarantees (\ref{boundest}) holding on $\omega_r$. Thin in turn secures
\begin{equation}
\frac{w(0',0,x_n)}{x_n} \leq  K(rM+C_r) \frac{\psi(0',0,x_n)-\psi(0)}{x_n}
\end{equation}
which after taking limit gives
\[(u_{(\xi)(\xi)})_{x_n}(0) \leq (C r M + C_r)\psi_n(0) \leq C r M + C_r.\]
\end{proof}

We now proceed to proving the main result advertised in the Introduction.

\begin{theorem}\label{complexKrylov}
For the solutions $u$ of the Dirichlet problem (\ref{ckns}) from Theorem \ref{krylovest} the estimate
\begin{equation}
\label{mainest}
|\nabla^2u|<C
\end{equation}
holds in $\overline{\Omega}$, where $C$ depends on $\Omega$, $n$, $|\phi|_{C^{3,1}}$ and $|f^{\frac 1 n}|_{C^{1,1}}$.
\end{theorem}
\begin{proof}
Suppose that $M$ is attained at $p\in \partial \Omega$. We choose adopted coordinates $z$ around this point in which $p=0$. We claim that for $u_{x_n}$ along the boundary $\partial \Omega$ around $0$ the inequality
\begin{equation} \label{ntaylor} u_{x_n}(z) \leq u_{x_n}(0)+a \cdot (z',y_n)+(rM+C_r)\xn \end{equation}
holds, where in the above $a$ and $C_r$ are bounded in terms of $(u_{x_n(\zeta_i)}(0))_{1 \leq i \leq 2n-1}$, $\gamma$ and previous $C_r$. Moreover $\zeta_i$ are the tangential projections of $\xi_i$ from Proposition \ref{approx} and $\gamma$ the interior normal. Thus the constants are bounded as a vector of mixed tangential-normal derivatives and $\gamma$ are under control.

Note that in order to obtain (\ref{ntaylor}), from Taylor expansion, we need an estimate on 
\[u_{x_n(\zeta_i)(\zeta_i)}(w)\]
along the boundary in the neighborhood $0$. Let us elaborate for clarity and completeness how to obtain that bound from Lemma \ref{tantann}.

First of all note that in the adopted coordinates $z$ for $p$
\begin{equation} \label{ngamma} u_{x_n}(z)\leq u_\gamma(z)+a\cdot(z',y_n)+C\xn\end{equation}
locally along $\partial \Omega$ for $a$ and $C$ under control, cf. (\ref{wdiff3}).

Writing out the Taylor expansion of $u_\gamma$ around $p$ on $\partial \Omega$ gives
\begin{equation} \label{gammataylor} u_\gamma(z) \leq u_\gamma(0)+a \cdot (z',y_n)+(rM+C_r)\xn \end{equation}
for $a=(u_{\gamma(\zeta_i)}(0))_{1 \leq i \leq 2n-1}$ provided
\begin{equation} u_{\gamma(\zeta_i)(\zeta_i)}(w) \leq rM+C_r \end{equation} for $w\in \partial \Omega$ around $p$.

Let us take the vector fields $\tilde{\xi}_i$ generated by Proposition \ref{approx} so that
\begin{equation} \label{taninf} \tilde{\xi}_i(w)=\zeta_i(w) \end{equation} in adopted coordinates $\tilde{z}$ for $w\in \partial \Omega$. Thanks to property (2) of Definition \ref{tang} for approximate tangential $\tilde{\xi}_i$'s we have
\begin{equation} \label{infin} u_{\tilde{x}_n(\tilde{\xi}_i)(\tilde{\xi}_i)}(w)=u_{\tilde{x}_n(\tilde{\zeta}_i)(\tilde{\zeta}_i)}(w),\end{equation}
where $\tilde{\zeta}_i$ are tangential projections of $\tilde{\xi}_i$'s.
Note that:
\begin{equation} \label{shuf} \begin{gathered}
u_{\tilde{x}_n(\tilde{\xi}_l)(\tilde{\xi}_l)}=u_{(\tilde{\xi}_l)(\tilde{\xi}_l)\tilde{x}_n}\\
-2 \Re \Big( \big((\tilde{\xi}_l)_i ((\tilde{\xi}_l)_j)_{\tilde{i}}\big)_{\tilde{x}_n} u_{\tilde{j}}\Big)\\
-2 \Re \Big( \big((\tilde{\xi}_l)_i (\tilde{\xi}_l)_j\big)_{\tilde{x}_n} u_{\tilde{i}\tilde{j}}
+\big((\tilde{\xi}_l)_i \overline{(\tilde{\xi}_l)_j}\big)_{\tilde{x}_n} u_{\tilde{i}\overline{\tilde{j}}}\Big). \end{gathered}
\end{equation}
Thus applying (\ref{shuf}) in (\ref{infin}) from Lemma \ref{tantann} and the fact that all the derivatives of $u$ apart from normal-normal ones are under control we obtain:
\begin{equation} \label{taylorcoefbound} u_{\tilde{x}_n(\tilde{\zeta}_i)(\tilde{\zeta}_i)}(w) \leq rM+C_r.\end{equation}
Another easy calculation demonstrates that
\begin{equation} \label{theyareclose} |u_{\tilde{x}_n(\tilde{\zeta}_i)(\tilde{\zeta}_i)}(w)-u_{\tilde{x}_n({\zeta_i})({\zeta_i})}(w)|<C.\end{equation}
The above follows from the expansion:
\begin{equation} \label{expanss}
u_{\tilde{x}_n({\zeta_l})({\zeta_l})}=u_{\tilde{x}_n[{\zeta_l}][{\zeta_l}]}
+ 2 \Re \Big( (\zeta_l)_i \big((\zeta_l)_j\big)_{\tilde{i}} u_{\tilde{x}_n\tilde{j}}+\overline{(\zeta_l)_{i}}\big((\zeta_l)_j\big)_{\bar{\tilde{i}}}u_{\tilde{x}_n\tilde{j}} \Big).
\end{equation}
Indeed, it shows that we only need to see why the terms in front of double normal derivatives in second summand of (\ref{expanss}) are the same for $\zeta_l$ and $\tilde{\zeta_l}$ at $w$. Since those are tangential vector fields we get
\begin{equation}\label{normaltangentialproduct}
(\zeta_l,\eta)=(\Re(\zeta_l),\gamma)=0=(\Re(\tilde{\zeta}_l),\gamma)=(\tilde{\zeta}_l,\eta).
\end{equation}
Using that at $w$
\begin{equation}
\label{normalatw}\gamma(w)=\partial_{\tilde{x}_n}
\end{equation}
and differentiating (\ref{normaltangentialproduct}) in any real $\tilde{a}$ direction (in $\tilde{z}$ coordinates) provides
\begin{equation}\label{byparts}
\Re \Big( \big((\zeta_l)_n\big)_{\tilde{a}}\Big)=-\Re \Big( (\zeta_l)_n\big(\overline{\eta_n}\big)_{\tilde{a}}\Big).
\end{equation} 
Note the equality
\begin{equation}\label{xnxnterm}
2 \Re \Big( (\zeta_l)_i \big((\zeta_l)_n\big)_{\tilde{i}}+\overline{(\zeta_l)_{i}}\big((\zeta_l)_n\big)_{\bar{\tilde{i}}} \Big) = 4\Re \Big( (\zeta_l)_i \big(\Re \left((\zeta_l)_n\right)\big)_{\tilde{i}} \Big)
\end{equation} 
and the fact that this is the quantity in (\ref{expanss}) standing in front of $u_{\tilde{x}_n \tilde{x}_n}$. Thus we see from (\ref{byparts}) that the quantity in (\ref{xnxnterm}) is identical for both vector fields due to (\ref{taninf}). Applying (\ref{theyareclose}) in (\ref{taylorcoefbound}) results in 
\[u_{\tilde{x}_n({\zeta_i})({\zeta_i})}(w) \leq rM+C_r.\]

Finally,
\begin{equation} \label{ngamma} u_{\gamma({\zeta_i})({\zeta_i})}(w) \leq C + u_{\tilde{x}_n({\zeta_i})({\zeta_i})}(w),\end{equation}
as another easy calculation exploiting (\ref{normalatw}) (and the bound on mixed derivatives) shows. This provides (\ref{gammataylor}). Merging (\ref{gammataylor}) with (\ref{ngamma}) proves (\ref{ntaylor}).

For the end of the proof consider the function
\[v(z)=K \big( -\psi(z) - \beta ||(z',y_n)||^2 - \frac{a}{K \cdot (rM+C_r)} \cdot (z',y_n) \big)\]
for some $K,\beta>0$.

Using (\ref{apsi}) and (\ref{seven}) as before, after choosing $\beta$ small enough in comparison to strict plurisubharmonicity of $-\psi$, $K$ big enough in comparison to $||f^{1/n}||_{C^2}$ we can secure
\begin{equation}\label{linfin}
\mathcal L \big( (rM+C_r)v + u_{x_n} - u_{x_n}(0)\big) \geq 0
\end{equation}
in $\Omega \cap B_r$ for $r<<1$.

On the other hand thanks to (\ref{ntaylor}) clearly 
\begin{equation}\label{fin}
u_{x_n} - u_{x_n}(0) \leq - (rM+C_r)v
\end{equation}
in $\partial \Omega\cap B_r$ while this holds on $\Omega \cap \partial B_r$ for $r<<1$ thanks to choosing $K\beta >>1$.

Using (\ref{linfin}) and (\ref{fin}) we end up with (\ref{fin}) holding in $\Omega\cap B_r$. Thus 
\begin{equation}\label{doubnormal}
M=u_{x_n x_n}(0) \leq (rM+C_r)v_{x_n}(0)=C(rM+C_r).
\end{equation} 
As we can choose $r<<1$ as small as we wish (\ref{doubnormal}) results in the desired bound on $M$.
\end{proof}

\section{An obstruction for general complex Hessian equations}
In this section we shall discuss the obstruction that occurs when a general complex Hessian equation replaces the complex Monge-Amp\`ere one.

As the discussion after (\ref{unitary}) points out for general $F$ the only symmetries one can use are the complex (i.e. unitary) affine transformations. In particular one cannot utilize the adapted coordinates from Section 5.

Nevertheless, let the boundary of $\Omega$ be given by a defining function $\rho$. After picking a boundary point and identifying it with the coordinate origin we have
\[ \rho(z)=2\Re(\rho_{z_j}(0)z_j)+\Re(\rho_{z_kz_l}(0)z_kz_l)+\rho_{z_k\bar{z}_l}(0)z_k\bar{z}_l+O(||z||^3).\]
In order to fix a normal direction one can apply a unitary rotation, if necessary, and assume that $\rho_{z_j}(0)=0,\ j=1,\cdots,n-1$ and $\rho_{z_n}(0)=-1$. Additionally, rotating further if necessary, we can assume that
\begin{equation}\label{condss}
 \rho_{z_kz_k}(0)\in\mathbb R,\ \ k=1,\cdots,n-1.
\end{equation}

Fix now $s\in\{1,\cdots,n-1\}$ and consider the holomorphic affine skew-hermitian vector field
$$\xi=\partial_{z_s}+a_{kl}z_l\partial_{z_k},$$
for some skew-hermitian matrix $(a_{kl})$. We would like to check to what extent the conditions in Definition \ref{tang} could be fulfilled. The condition
$$(\xi,\eta)=\Re\left(\xi_k(z)\rho_{z_k}(z)\right)=O(||(z',y_n)^2||)$$
(which boils down to canceling all linear terms, except for $\Re(z_n)$) is then easily transformed into the linear system.
\begin{equation}\label{obstr1}
\begin{cases}
 a_{nl}=\rho_{z_sz_l}(0)+\rho_{\bar{z}_sz_l}(0),\ &l=1,\cdots n-1;\\
 a_{nn}=i\Im(\rho_{\bar{z}_sz_n}(0)+\rho_{z_sz_n}(0)).
\end{cases}
\end{equation}
Define then $a_{ln}:=-\overline{a}_{ln}$ and $a_{pq}=0$ for $1\leq p,q\leq n-1$.

It remains to check that such a $\xi$ satisfies $||\xi||^2=1+O(||(z',y_n)^2||)$.
But
$$||\xi||^2=\big|1-\big(\rho_{\bar{z}_s\bar{z}_s}(0)+\rho_{{z}_s\bar{z}_s}(0)\big)z_n\big|^2+O(||(z',y_n)^2||)$$
$$=1-2\Re\Big(\big(\rho_{\bar{z}_s\bar{z}_s}(0)+\rho_{{z}_s\bar{z}_s}(0)\big)z_n\Big)+O(||(z',y_n)^2||)$$
$$=1-2\big(\rho_{\bar{z}_s\bar{z}_s}(0)+\rho_{{z}_s\bar{z}_s}(0)\big)\Re(z_n)+O(||(z',y_n)^2||)=1+O(||(z',y_n)^2||),$$
where we used (\ref{condss}). Thus, $\xi$ is an almost tangential skew-hermitian holomorphic vector field.

For $\xi=i\partial_{z_s}+a_{kl}z_l\partial_{z_k}$ (and still $s\in1,\cdots,n-1$) the same argument leads to the (unique) matrix
$$\begin{cases}
   a_{pq}=0\ & {\rm if}\ 1\leq p,q\leq n-1;\\
   a_{nl}=i\rho_{z_sz_l}(0)-i\rho_{\bar{z}_sz_l}(0)\ &{\rm if}\ 1\leq l\leq n-1;\\
   a_{ln}=-\overline{a}_{nl}\\
   a_{nn}=i\Re\big(\rho_{z_sz_n}(0)-\rho_{\bar{z}_sz_n}(0)\big).
  \end{cases}
$$
Once again $||\xi||^2=1+O(||(z',y_n)^2||)$ thanks to (\ref{condss}).

For the remaining direction $i\partial_{z_n}$ and the corresponding vector field $\xi=i\partial_{z_s}+a_{kl}z_l\partial_{z_k}$ the unique choice for the matrix $a$ is

$$\begin{cases}
   a_{pq}=0\ & {\rm if}\ 1\leq p,q\leq n-1;\\
   a_{nl}=i\rho_{z_nz_l}(0)-i\rho_{\bar{z}_nz_l}(0)\ &{\rm if}\ 1\leq l\leq n-1;\\
   a_{ln}=-\overline{a}_{nl}\\
   a_{nn}=i\big(\Re\rho_{z_nz_n}(0)-\rho_{\bar{z}_nz_n}(0)\big).
  \end{cases}
$$
With such a choice
$$||\xi||^2$$
$$=\Big|i+i\sum_{l=1}^{n-1}\big(\rho_{z_nz_l}(0)-\rho_{\bar{z}_nz_l}(0)\big)z_l+i\big(\Re \big(\rho_{z_nz_n}(0)\big)-\rho_{\bar{z}_nz_n}(0)\big)z_n\Big|^2+O(||(z',y_n)^2||)$$
$$=1+2\Re\sum_{l=1}^{n-1}\big(\rho_{z_nz_l}(0)-\rho_{\bar{z}_nz_l}(0)\big)z_l+O(||(z',y_n)^2||).$$

But the latter quantity is $O(||(z',y_n)^2||)$ iff 
\[\rho_{z_nz_l}(0)-\rho_{\bar{z}_nz_l}(0)=-i\rho_{y_nz_l}(0)\]
vanishes for each $l=1,\cdots,n-1$.

Thus, in general, there is no almost tangential skew-hermitian holomorphic vector field whose real part extends $\partial_{y_n}$.

As a final remark, let us mention that for $\Omega$ being a ball in $\mathbb C^n$ the obstructions $\rho_{y_nz_l}(0)$ do vanish. In fact, then the constructed vector fields are just the standard infinitesimal rotations which are tangential to the boundary. Hence, we finish with the following corollary:
\begin{corollary}
 Let $F$ be a complex Hessian operator satisfying the conditions (\ref{below}). If $B\subset\mathbb C^n$ is a ball, then the Dirichlet problem
 $$\begin{cases}
    F(D^2_{\mathbb C}u)={f},\\
    u|_{\partial B}=\varphi
   \end{cases}
$$
admits a unique $\Gamma$-admissible solution of class $C^{1,1}(\overline{B})$ provided $\varphi\in C^{3,1}(\partial B)$ and {$f\in C^{1,1}(B)$}, {$f \geq 0$}.
\end{corollary}

	\bibliographystyle{amsplain}
	
\end{document}